\documentclass[11pt]{article}
\usepackage[utf8]{inputenc}
\usepackage[english]{babel}
\usepackage{epsfig,enumerate,amsmath,amsfonts,amsbsy,amssymb,amsthm,mathrsfs,lineno,ifpdf}
\usepackage{indentfirst,relsize,capt-of}
\usepackage[mathscr]{euscript}
\usepackage[numbers]{natbib}
\usepackage{setspace,graphicx}
\usepackage[normalem]{ulem}
\usepackage{latexsym}
\usepackage[usenames,dvipsnames]{pstricks}
\usepackage[color=blue!20,textsize=small,textwidth=2.2cm]{todonotes}

\usepackage{pst-grad} 
\usepackage{pst-plot} 
\usepackage{tikz,pgf} 
\usetikzlibrary{positioning,shapes,shadows,arrows} 
\usepackage[lined,ruled,linesnumbered]{algorithm2e}
\usepackage[margin=1.2in]{geometry}
\newcounter{dummy} \numberwithin{dummy}{section}
\newtheorem{theorem}[dummy]{Theorem}
\newtheorem{lemma}[dummy]{Lemma}

\newtheorem{coro}[dummy]{Corollary}
\newtheorem{obs}[dummy]{Observation}

\newtheorem{conj}[dummy]{Conjecture}

\newtheorem{claim}{Claim}

\usepackage{enumitem}
\usepackage{etoolbox}

\newtheorem{manualtheoreminner}{Theorem}
\newenvironment{manualtheorem}[1]{%
  \renewcommand\themanualtheoreminner{#1}%
  \manualtheoreminner
}{\endmanualtheoreminner}

\let\oldenumerate\enumerate
\renewcommand{\enumerate}{
	\oldenumerate
	\setlength{\itemsep}{1.5pt}
	\setlength{\parskip}{0pt}
	\setlength{\parsep}{0pt}
}
\numberwithin{equation}{section}

\usepackage{hyperref, url}
\usepackage{graphics}
\usepackage[most]{tcolorbox}
\hypersetup{
	colorlinks=true,
	linkcolor=blue,
	pdfpagemode=FullScreen,
}
\newtcolorbox{mytextbox}[1][]{%
	sharp corners,
	enhanced,
	colback=white,
	height=10cm,
	attach title to upper,
	#1
}

\graphicspath{{images/}{../images/}}

\begin{document}
	\title{Cops and robber on subclasses of $P_5$-free graphs}

\author{$^1$Uttam K. Gupta, $^2$Suchismita Mishra, $^1$Dinabandhu Pradhan\thanks{Corresponding author.}  \\ \\
		$^{1}$Department of Mathematics \& Computing\\ Indian Institute of Technology (ISM), Dhanbad\\ \\
		$^{2}$ Department of Mathematics, Universidad Andr\'es Bello, Chile\\
		\small \tt Email: ukumargpt@gmail.com; suchismitamishra6@gmail.com; dina@iitism.ac.in}
	
	\date{}
	\maketitle

\begin{abstract}
				
The game of cops and robber is a turn based vertex pursuit game played on a connected graph between a team of cops and a single robber. The cops and the robber move alternately along the edges of the graph. We say the team of cops win the game if a cop and the robber are at the same vertex of the graph. The minimum number of cops required to win in each component of a graph is called the cop number of the graph. Sivaraman [Discrete Math. 342(2019), pp. 2306-2307] conjectured that for every $t\geq 5$, the cop number of a connected $P_t$-free graph is at most $t-3$, where $P_t$ denotes a path on $t$~vertices. Turcotte [Discrete Math. 345 (2022), pp. 112660] showed that the cop number of any $2K_2$-free graph is at most $2$, which was earlier conjectured by Sivaraman and Testa. Note that if a connected graph is $2K_2$-free, then it is also $P_5$-free. Liu showed that the cop number of a connected ($P_t$, $H$)-free graph is at most $t-3$, where $H$ is a cycle of length at most $t$ or a claw. So the conjecture of Sivaraman is true for ($P_5$, $H$)-free graphs, where $H$ is a cycle of length at most $5$ or a claw. In this paper, we show that the cop number of a connected ($P_5,H$)-free graph is at most $2$, where $H\in \{C_4$, $C_5$, diamond, paw, $K_4$, $2K_1\cup K_2$, $K_3\cup K_1$, $P_3\cup P_1\}$. 
\end{abstract}
	\noindent
	{\small \textbf{Keywords:} Cops and Robber; cop number; forbidden induced subgraphs; $P_5$-free graphs.}
	
\section{Introduction}

All the graphs in this paper are finite, simple, and undirected. The complete graph, cycle, and path on $n$~vertices are denoted by $K_n,C_n$, and $P_n$, respectively. The \emph{disjoint union} of two vertex-disjoint graphs $G$ and $H$, denoted by $G \cup H$, is the graph with vertex set $V(G) \cup V (H)$ and edge set $E(G) \cup E(H)$. For a positive integer~$r$, $rG$ denotes the disjoint union of $r$~copies of $G$. The game of cops and robber was introduced by Quilliot \cite{quilliot} in~1978 and independently by Nowakowski and Winkler \cite{nowakowskivtovpursuit} in~1981. The game is played on a connected graph by a team of cops and a robber. In the first turn, all the cops are placed on the vertices of the graph (multiple cops can be placed on a single vertex). In the second turn, the robber chooses a vertex. Then the cops and the robber take their moves in alternative turns, starting with the cops.  A valid move for a cop is to stay at its current position or to move to an adjacent vertex. A valid move for the robber is similar to the cops. A \emph{round} of moves consists of  two consecutive turns in which the cops have the first turn and then the robber has its turn. We say that a cop {\it captures} the robber if both of them are on the same vertex of the graph. The cops {\it win} if after a finite number of rounds, one of the cops captures the robber. The robber {\it wins} if the cops cannot win in a finite number of rounds.
	The \emph{cop number} of a graph $G$, denoted by $cop(G)$, is defined as the minimum number of cops required such that the cops win in each component of $G$.
	
	The characterization of graphs with the cop number one was studied by Quilliot \cite{quilliot} and independently by Nowakowski and Winkler \cite{nowakowskivtovpursuit}. The study of graphs with higher cop number was initiated by Aigner and Fromme \cite{aigner}. Several characterizations of graphs with cop number $k$ have been studied by Clarke and MacGillivray \cite{charKCopWin}. Berarducci and Intrigila~\cite{berarduccioncopnumbers} gave an $O(n^{O(k)})$-algorithm to decide whether $cop(G)\leq k$ for a $n$-vertex graph $G$. Fomin et al.~\cite{fomin} showed that it is NP-hard to determine the cop number of a graph. Moreover, it is $W[2]$-hard to determine whether $cop(G)\leq k$, where $k$ is the parameter \cite{fomin}. Meyniel conjectured that the cop number of a connected graph on $n$~vertices is at most $O(\sqrt{n})$. The conjecture was made by Meyniel in a personal communication with Frankl in~$1985$ and is mentioned in~\cite{frankl}. Frankl showed that $cop(G) = o(n)$ for any connected $n$-vertex graph $G$. One can ask whether there exists $\epsilon >0$ such that the cop number of any connected $n$-vertex graph $G$ is $O(n^{1-\epsilon})$? This problem is also open. For more details, see the survey by Baird and Bonato \cite{baird}. The best known upper bound on the cop number of a graph on $n$~vertices is $n2^{-(1+o(1))\sqrt{\log_2 n}}$ (see \cite{lu,scott2011}). Lu and Peng \cite{lu} showed that if $G$ is a graph of diameter at most~$2$ or a bipartite graph of diameter at most~$3$, then $cop(G)\leq 2\sqrt{n}-1$; thus validating the Meyniel's conjecture for such graphs. Later, Wagner \cite{wagner} improved the result of Lu and Peng on the same class of graphs by showing that such graphs satisfy $cop(G)\leq \sqrt{2n}$. Lu and Peng used random arguments to prove the bound whereas Wagner's proof did not include randomness. 
	
	The cop number of a family of graphs $\mathcal{G}$ is the minimum integer $k$ such that $cop(G) \leq k$ for any graph $G$ in $\mathcal{G}$. If we cannot find such an integer $k$, then we say $\mathcal{G}$ is a family with unbounded cop number (or the cop number of $\mathcal{G}$ is not bounded). Aigner and Fromme \cite{aigner} proved that the cop number of a graph with girth $\ell$ is at least its minimum degree, when $\ell>4$. Thus the cop number of the class of all graphs is not bouned. Andreae \cite{tandreaepursuitofgame} showed the cop number of the class of all $d$-regular graphs is not bounded. Aigner and Fromme \cite{aigner} proved that the cop number of the class of all planar graph is $3$. Moreover, dodecahedron is a planar graph whose cop number is $3$. Clarke~\cite{clarkeouterplanar} proved that the cop number of any outerplanar graph is at most~$2$. The cop number of the class of all $k$-chordal graphs is at most $k-1$ \cite{chordal}. 
	
	Let $\mathcal{F}$ be a family of graphs. A graph $G$ is said to be $\mathcal{F}$-free if no $H \in \mathcal{F}$ is isomorphic to an induced subgraph of the graph $G$. When $\mathcal{F}=\{H\}$ (resp. $\{H_1,H_2, \ldots, H_k\}$, $k\geq 2$), we use $H$-free (resp. $(H_1,H_2,\ldots, H_k)$-free) graphs to denote the  $\mathcal{F}$-free graphs. Joret et al.~\cite{gjoretforbiddengraphs} proved that the cop number of the class of $H$-free graphs is bounded if and only if every component of $H$ is a path. Furthermore, suppose that $\mathcal{G}$ be a family of graphs such that the diameter of any $G \in \mathcal{G}$ is at most $k$ for some natural number $k$. Masjoodi and Stacho~\cite{mmasjoodyforbiddensubgraphs} showed that the class of all $\mathcal{G}$-free graphs has bounded cop number if and only if $\mathcal{G}$ contains a path or $\mathcal{G}$ contains a generalized claw and a generalized net. It is known that the cop number of a connected $P_t$-free graphs is at most $t-2$ \cite{gjoretforbiddengraphs}. 
	Sivaraman \cite{vaidyapplicationtogpath} gave a shorter proof for the same by using Gy\'arf\'as path argument, where he conjectured the following.
	
	\begin{conj}[\cite{vaidyapplicationtogpath}]\label{conj}
		The cop number of a connected $P_t$-free graph is at most $t-3$ for $t\geq 5$.
	\end{conj}

Conjecture \ref{conj} was verified for a subclass of $P_t$-free graphs; the class of ($P_t, C_\ell$)-free graphs, where $t \geq 5$ and $3 \leq \ell \leq t$ \cite{MLiuforbiddesgrph}. 
	It is known that the cop number of a $P_4$-free graph is at most $2$ \cite{gjoretforbiddengraphs}. Note that if a graph is $2K_2$-free, then it is also $P_5$-free.
	Sivaraman and Testa \cite{vaidy2k2freegraphs} showed that the cop number of a connected $2K_2$-free graph $G$ is at most~$2$ if the diameter of $G$ is $3$ or it is $C_\ell$-free for some $\ell \in \{3,4,5\}$. In the same paper, they conjectured that the class of $2K_2$-free graphs has cop number~$2$. Later, this conjecture was proved by Turcotte~\cite{jturcotte2k2freegraphs}.

\subsection{Our results}	
Liu~\cite{MLiuforbiddesgrph} showed that the cop number of a connected ($P_t, H$)-free graph is at most $t-3$, where $t \geq 5$ and $H$ is a cycle of length at most~$t$ or a claw. First we give a short proof for these results for the case $t=5$ by using Gy\'arf\'as path argument (see Lemma~\ref{thmc4free} of Section~\ref{prelim}). Then by using that, we find the cop number of different subclasses of $P_5$-free graphs. In Section~\ref{sec:pawk4}, we first show that the cop number of a connected ($P_5$, paw)-free graph is at most $2$. So if a ($P_5, K_4$)-free graph is also paw-free, then its cop number is at most~$2$. We study the structure of a connected ($P_5, K_4$)-free graph, that has an induced paw, around an induced paw. By using the structural properties, we show that the cop number of a connected ($P_5, K_4$)-free graph is at most $2$ in Section~\ref{sec:pawk4}. We show that the cop number of every connected ($P_5,K_3 \cup K_1)$-free graph and every connected $P_3 \cup P_1$-free graph is at most $2$ in Section~\ref{sec:k3Uk1} and Section~\ref{sec:P3P1}, respectively. In Section~\ref{sec:diamondand2k1Uk2}, we show that the cop number of every connected ($P_5$, diamond)-free graph and every connected $(P_5,2K_1\cup K_2)$-free graph is at most $2$. We prove these results by showing the non-existence of any minimum counterexample. Again, $C_4$ and $C_5$ are $(P_4,P_3\cup P_1,2K_2)$-free and ($P_5,K_4$, diamond, paw, $C_4$, claw, $K_3 \cup K_1,2K_1 \cup K_2$)-free graphs, respectively with the cop number $2$. Therefore, the cop number of the class of ($P_5, H$)-free graphs is~$2$, where $H$ is a graph on $4$ vertices and has at least one edge.
	
\section{Notations, terminologies, and preliminary results}\label{prelim}
	
Let $G$ be a graph and $x$ be a vertex of $G$. The \emph{neighborhood} of $x$ in $G$, denoted by $N(x)$, is the set of all the neighbors of $x$ in $G$. The \emph{closed neighborhood} of $x$ in $G$, denoted by $N[x]$, is the set $\{x\}\cup N(x)$. For a set $S\in V(G)$, we define $N(S)=\{v\in V(G)\setminus S \mid N(v)\cap S\neq \emptyset\}$ and $N[S]=S\cup N(S)$. A set $D\subseteq V(G)$ is  called a \emph{dominating set} of $G$ if $N[D]=V(G)$. For two disjoint sets of vertices $S$ and $T$, $[S,T]$ denotes the set $\{xy\in  E(G) \mid  x \in S, y \in T\}$. We say that $[S,T]$ is \emph{complete} if every vertex in $S$ is adjacent to every vertex of $T$ in $G$. For a set of vertices $S$, $G[S]$ and $G-S$ denote the subgraphs induced by $S$ and $V(G)\setminus S$ in $G$, respectively. The length of a path is the number of edges in it. The distance of a vertex $u$ from a set $S$ is the least length of $u-v$ paths for every vertex $v \in S$. We use Cop~$1$ and Cop~$2$ as the names of the two cops throughout the paper. We refer to Figure~\ref{special} for some special graphs mentioned in this paper.
	
\begin{figure}[hbtp]
\centering
\includegraphics[scale=1]{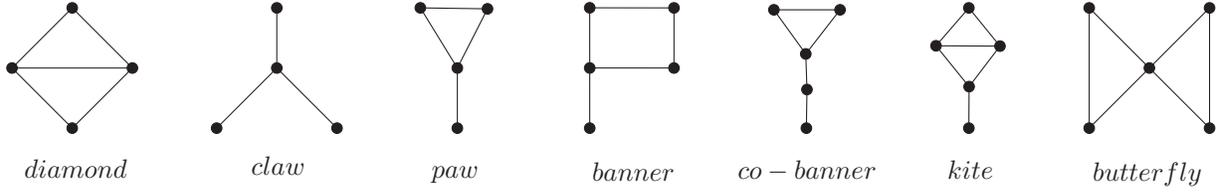}
\caption{Some special graphs}\label{special}
\end{figure}

Recall that in the game of cops and robber, a \emph{round} of moves consists of  two consecutive turns in which the cops have the first turn and then the robber has its turn. Moreover, a cop captures the robber if they are at the same vertex of the graph. This happens only if in the robber's turn, the robber stays or moves to a position that belongs to the closed neighborhood of the position of one of the cops. Equivalently, we can say that a cop at $x$ captures the robber at $y$ if after any round of moves of the cops and the robber, we have $y\in N[x]$.

Liu~\cite{MLiuforbiddesgrph} proved that the cop number of any connected $(P_t, C_\ell)$-free graph is at most $t-3$ for any natural number $t \geq 5$ and $\ell \leq t$. Moreover, it also has been shown that the cop number of a connected ($P_t$, claw)-free graph is at most $t-3$ for $t\geq 5$. To make the paper self-contained, we prove these results for $t=5$. 
\begin{lemma}[\cite{MLiuforbiddesgrph}]\label{thmc4free}
		Let $G$ be a connected $(P_5,H)$-free graph, where $H\in \{C_3,C_4,C_5,$ claw$\}$. Then $cop(G)\leq 2$.
	\end{lemma}	
	\begin{proof}
We prove the lemma by arguing that two cops capture the robber after a finite number of turns. Let $v_1$ be a vertex of $G$. In the first turn, we place both the cops at $v_1$. To avoid immediate capture, the robber must choose a vertex $y$ that is not a neighbor of $v_1$. Let $X$ be the component of the graph induced by $V(G) \setminus N[v_1]$ such that $y\in X$. Since $G$ is connected, we have $N(v_1)\cap N(X)\neq \emptyset$. In the next turn, Cop~$1$ stays at $v_1$ and Cop~$2$ moves to a vertex $v_2 \in N(v_1) \cap N(X)$. If the robber moves to a vertex of $N(v_1)$, then the cop at $v_1$ captures the robber. So to avoid capture, the robber should stay in $X$. Let the robber choose $x\in N[y]\cap X$ as its position. If $x$ is a neighbor of $v_2$, then the robber gets captured by Cop~$2$. So we may assume that $x$ is at distance at least $2$ from $v_2$. Since $v_2$ has a neighbor in $X$ and $X$ is a component of the graph induced by $V(G) \setminus N[v_1]$, there exists an induced path $P$ between $v_2$ and $x$ of length at least~$2$ such that $V(P)\setminus \{v_2\} \subseteq V(X)$. This implies that $V(P) \cup \{v_1\}$ induces a path of length at least~$3$. Since $G$ is $P_5$-free, the length of the path induced by $V(P)\cup \{v_1\}$ is $3$. This implies that $v_2$ is at distance $2$ from $x$. Let $v_3 \in V(P)\cap X$ be a common neighbor of $v_2$ and $x$. Note that $\{v_1,v_2,v_3,x\}$ induces a $P_4$ in $G$. Define $S := \{u \in V(G) \mid N(u) \cap \{v_1,v_2,v_3,x\} = \{v_1,x\}\}$ and $T := \{u \in V(G)\setminus \{v_3\} \mid N(u) \cap \{v_1,v_2,v_3,x\} = \{v_2,x\}\}$. 
		
We first prove that if $S=\emptyset$ or $T=\emptyset$, then $cop(G)\leq 2$. Suppose that $S=\emptyset$. Then in the next turn, Cop~$1$ and Cop~$2$ move to $v_2$ and $v_3$, respectively. To avoid immediate capture, the robber must move. Suppose that the robber moves to a vertex $r$. Since $S = \emptyset$ and $\{v_1,v_2,v_3,x,r\}$ does not induce a $P_5$, $r$ is adjacent to $v_2$ or $v_3$. Hence the robber gets captured by Cop~$1$ or Cop~$2$ implying that $cop(G)\leq 2$. Now suppose that $T=\emptyset$. In the next turn, Cop~$1$ stays at $v_1$ and Cop~$2$ moves to $v_3$, respectively. To avoid immediate capture, the robber must move to a vertex $r$. Since $T = \emptyset$ and $\{v_1,v_2,v_3,x,r\}$ does not induce a $P_5$, $r$ is adjacent to $v_1$ or $v_3$. Hence the robber gets captured by Cop~$1$ or Cop~$2$ implying that $cop(G)\leq 2$. Note that if $H\in \{C_4,$ claw$\}$, then $T=\emptyset$. Again note that if $H=C_5$, then $S=\emptyset$. So $cop(G)\leq 2$ if $H\in\{C_4,C_5,$ claw$\}$. 

Now assume that $S\neq \emptyset$ and $T\neq \emptyset$. Then it is clear that $H=C_3$. Let $s$ be a vertex of $S$. In the next turn, Cop~$1$ stays at $v_1$ and Cop~$2$ moves to $v_3$. To avoid immediate capture, the robber must move to a vertex $r$ that is not adjacent to $v_1$ and $v_3$. Note that $r\in T$; otherwise $\{v_1,v_2,v_3,x,r\}$ induces a $P_5$. Again, $r$ is not adjacent to $s$; otherwise $\{s,x,r\}$ induces a $C_3$. In the next turn, Cop~$1$ moves to $v_2$ and Cop~$2$ stays at $v_3$. To avoid immediate capture, the robber must move to a vertex $r'$ that is not adjacent to $v_2$ and $v_3$. Since $\{r',r,x\}$ does not induce a $C_3$, $r'$ is not adjacent to $x$. Again since $\{s,v_1,v_2,r,r'\}$ does not induce a $P_5$, $r'$ is adjacent to $s$ or $v_1$. If $r'$ is adjacent to $s$, then $\{r',s,x,v_3,v_2\}$ induces a $P_5$ which is a contradiction. So $r'$ is adjacent to $v_1$. Now $\{r',v_1,v_2,v_3,x\}$ induces a $P_5$, again a contradiction. So such a vertex $r'$ does not exist. Hence the robber cannot escape from $r$ and gets captured implying that $cop(G)\leq 2$.\end{proof}

Note that by Lemma~\ref{thmc4free}, every $(P_5,C_4)$-free graph has the cop number at most~$2$. A more careful observation on the proof of Lemma~\ref{thmc4free} leads to a result on a superclass of the class of $(P_5,C_4)$-free graphs. Note that in the proof of Lemma~\ref{thmc4free}, if $T\neq\emptyset$, then for any $t\in T$, $\{v_1,v_2,v_3,x,t\}$ induces a banner in $G$. So the following corollary holds.

\begin{coro}\label{p5banner}
Let $G$ be a connected $(P_5,$ banner$)$-free graph. Then $cop(G)\leq 2$.
\end{coro}

We now conclude this section by describing a partition around an induced paw of a graph to use that later. The partition is described as follows.
\begin{itemize}
%
%
%
 \item\underline{\textsc{partition around an induced paw}}: Let $G$ be a graph and $P$ induce a paw in $G$ with vertex set $P=\{v_1,v_2,v_3,v_4\}$ and edge set $\{v_1v_2,v_2v_3,v_3v_4,v_2v_4\}$. Define the following sets.
 		
	\[	\begin{array}{ll}\tag{$\mathcal{P}_1$}
		\label{pawstruct1}
			A_{i}:=&\{v\in N(P)\mid N(v)\cap P=\{v_i\}\}, 1 \leq i \leq 4\\
			B_{ij}:=&\{v\in N(P)\mid N(v)\cap P=\{v_i,v_j\}\}, 1 \leq i < j \leq 4\\
			T_i:=&\{v\in N(P) \mid N(v)\cap P= P \setminus \{v_i\}\}, 1 \leq i \leq 4\\ 
			D:=&\{v\in N(P) \mid P\subseteq N(v)\}\\
			X:=&V(G)\setminus (P\cup N(P))
	\end{array}\]	
\end{itemize}

Let $A=\bigcup\limits_{1 \leq i \leq 4}A_{i}$, $B = \bigcup\limits_{1 \leq i <j \leq 4}B_{ij}$ and $T = \bigcup\limits_{1 \leq i \leq 4}T_i$. Note that $(A, B, D, P, T, X)$ is partition of $V(G)$ when the structure of $G$ is considered around an induced paw.

	
%

\section{On the class of ($P_5$, paw)-free graphs and ($P_5,K_4$)-free graphs}\label{sec:pawk4}


In this section, we first show that the cop number of a $(P_5$, paw$)$-free graph is at most $2$ (Lemma~\ref{pawlemma}). Then we use this result to show that the cop number of any connected ($P_5, K_4$)-free graph $G$ is at most~$2$. The idea of the proof is as follows. If $G$ is paw-free, then by Lemma~\ref{pawlemma}, $cop(G)\leq 2$. If $G$ contains an induced paw, then we consider the presence of an induced co-banner in $G$. If $G$ contains an induced co-banner, then we consider the partition of $V(G)$ as defined in \ref{pawstruct1} around the paw contained in an induced co-banner of $G$. Then we have $A_1\neq \emptyset$ which help us to prove that the robber gets captured by two cops (Lemma~\ref{lem:K4FreeWithCobanner}). On the other hand, if $G$ is co-banner-free, then we consider the presence of an induced butterfly in $G$. If $G$ contains an induced butterfly, then we consider the partition of $V(G)$ as defined in \ref{pawstruct1} around a paw contained in an induced butterfly of $G$. Then we have $B_{12}\neq \emptyset$ which help us to prove that the robber gets captured by two cops (Lemma~\ref{lem:K4FreeWithButterfly}). If $G$ is butterfly-free, then we consider the partition of $V(G)$ as defined in \ref{pawstruct1} around any induced paw of $G$. Since $G$ is (co-banner, butterfly)-free, we have $A_1=B_{12}=\emptyset$. In fact, the intuition of the proof is to gradually prove that $A_1=B_{12}=\emptyset$ while taking the stronger hypothesis that $G$ is co-banner-free and butterfly-free. Finally we show that the robber gets captured by two cops (Lemma~\ref{lem:K4FreeWithoutkite}).

	\begin{lemma} \label{pawlemma}
		Let $G$ be a connected $(P_5,paw)$-free graph. Then $cop(G) \leq 2$.
	\end{lemma}
	
	\begin{proof}
		If $G$ is $C_3$-free, then by Lemma~\ref{thmc4free}, $cop(G) \leq 2$. So we may assume that $G$ contains a $C_3$, say with vertex set $K = \{u_1, u_2, u_3\}$. If for any vertex $x\in N(K)$, $|N(x)\cap K|=1$, then $K\cup \{x\}$ induces a paw, a contradiction. So every vertex of $N(K)$ is adjacent to at least two vertices of $K$. Now we show that every vertex of $G$ has a neighbor in $K$. If possible, then let $y$ be a vertex at distance~$2$ from $K$. Let $w$ be a neighbor of $y$ in $N(K)$. Now since $w\in N(K)$, $w$ is adjacent to at least two vertices of $K$. Without loss of generality, we may assume that $w$ is adjacent to $u_1$ and $u_2$. Then $\{u_1,u_2,w,y\}$ induces a paw, a contradiction. So we may conclude that such a vertex $y$ does not exist. Thus no vertex of $G$ is at distance~$2$ from $K$. Since $G$ is connected, every vertex of $G$ has a neighbor in $K$. Notice that $\{u_1,u_2\}$ is a dominating set of $G$. Therefore, $cop(G)\leq 2$.  
	\end{proof}

Let $G$ be a connected $(P_5,K_4)$-free graph containing an induced paw, say with vertex set $P=\{v_1,v_2,v_3,v_4\}$ and edge set $\{v_1v_2, v_2v_3, v_3v_4,v_2v_4\}$. Define sets $A_{i},B_{ij}, T_i, D,$ and $X$ around $P$ as defined in~\ref{pawstruct1} for every $1 \leq i,j \leq 4$ and $i<j$. Since $G$ is $K_4$-free, it can be noticed that $T_1=D=\emptyset$. So we have the following observation. 
	
\begin{obs}\label{i8}
		$T_1=D=\emptyset$.		
\end{obs}
	
\begin{obs}\label{i9} The following hold. 
\begin{enumerate}
\item $N(X) \cap N(P) \subseteq A_2\cup B_{12}\cup B_{23}\cup B_{24}\cup T_2\cup T_3\cup T_4$.
\item For any vertex $v\in N(P)$ and any component $H$ of $G[X]$, either $[\{v\},V(H)]$ is complete or $[\{v\},V(H)]=\emptyset$.
\item Every vertex in $X$ is at distance $2$ from $P$.
\end{enumerate}
\end{obs}
\begin{proof}
(a) Let $u$ be an arbitrary vertex of $N(X) \cap N(P)$ and $x$ be a neighbor of $u$ in $X$. Since $\{x,u,v_1,v_2,v_3\}$ does not induce a $P_5$, $u\notin A_1 \cup A_3 \cup B_{14} \cup B_{34}$. Again since $\{x,u,v_4,v_2,v_1\}$ does not induce a $P_5$, we have $u\notin A_4 \cup B_{13}$. By Observation~\ref{i8}, $T_1=D=\emptyset$. Therefore, $u \in  A_2\cup B_{12}\cup B_{23}\cup B_{24}\cup T_2\cup T_3\cup T_4$. Since $u$ is an arbitrary vertex of $N(X) \cap N(P)$, (a) holds.
\smallskip

(b) If $[\{v\},V(H)]=\emptyset$, then we are done. Suppose that $[\{v\},V(H)]\neq\emptyset$. Let $x$ be a neighbor of $v$ in $H$. First we show that $v$ is adjacent to every vertex of $N(x)\cap X$. For the sake of contradiction, let $x'\in N(x)\cap X$ such that $x'$ is not adjacent to $v$. Since $v\in N(X)\cap N(P)$, by (a), $v\in  A_2\cup B_{12}\cup B_{23}\cup B_{24}\cup T_2\cup T_3\cup T_4$. Note that there exist distinct vertices $v_i,v_j\in P$ for $i,j \in \{1,2,3,4\}$ such that $vv_i, v_iv_j \in E(G)$ and  $vv_j\notin E(G)$. Then $\{v_j,v_{i},v,x,x'\}$ induces a $P_5$, a contradiction. So $v$ is adjacent to every vertex of $N(x)\cap X$. Now we show that $v$ is adjacent to every vertex of $H$. Let $y$ be an arbitrary vertex of $V(H)\setminus \{x\}$. If $y$ is adjacent to $x$, then, since $v$ is adjacent to every vertex of $N(x)\cap X$, $v$ is adjacent to $y$. If $y$ is not adjacent to $x$, then let $xx_1x_2\ldots x_ky;k\geq 1$ be an induced path between $x$ and $y$ in $H$. By our argument, we have that $v$ is adjacent to $x_1$. Again by our argument for $x_1$, we have that $v$ is adjacent to $x_2$. Following this way, we conclude that $v$ is adjacent to $y$. Now since $y$ is an arbitrary vertex of $V(H)\setminus \{x\}$, $v$ is adjacent to every vertex of $H$. Thus (b) holds.
\smallskip

(c) Let $H$ be a component of $G[X]$. Since $G$ is connected, we have $N(V(H))\cap N(P)\neq \emptyset$. Let $v\in N(V(H))\cap N(P)$. By (b), every vertex of $H$ is adjacent to $v$. This implies that every vertex of $H$ is at distance~$2$ from $P$. Since $H$ is an arbitrary component of $G[X]$, (c) holds.
\end{proof}

  In the following lemma, we show that the cop number of a connected ($P_5,K_4$)-free graph that has an induced co-banner is at most~$2$.

\begin{lemma}\label{lem:K4FreeWithCobanner}
	Let $G$ be a connected $(P_5,K_4)$-free graph. If $G$ has an induced co-banner, then $cop(G) \leq 2$.
\end{lemma}
\begin{proof}
	Suppose that $\{a_1,v_1,v_2,v_3,v_4\}$ induces a co-banner in $G$ with edge set $\{a_1v_1,v_1v_2,v_2v_3,$ $v_3v_4,v_4v_2\}$. Note that $P=\{v_1,v_2,v_3,v_4\}$ induces a paw in $G$. Define the sets $A_i,B_{ij},T_i,D$, and $X$ around $P$ as defined in \ref{pawstruct1} for every $1 \leq i,j \leq 4$ and $i<j$. Note that $a_1\in A_1$. Again, $A_3=\emptyset$; otherwise for any $u \in A_3$, either $\{a_1,v_1,v_2,v_3,u\}$ or $\{u,a_1,v_1,v_2,v_4\}$ induces a $P_5$. Similarly, we can show that $A_4=\emptyset$. In the first turn, we place Cop~$1$ and Cop~$2$ at $v_1$ and $v_2$, respectively. To avoid immediate capture, the robber should choose a vertex $x$ that is not adjacent to $v_1$ and $v_2$. Since $A_3=A_4 = \emptyset$, we have $x \in B_{34}\cup X$. To proceed further, we first prove a series of claims.
	
\begin{claim}\label{lem:K4FreeWithCobanner-cl0}
	$N(X) \cap N(P)\subseteq A_2\cup B_{12}\cup T_2\cup T_3\cup T_4$.
	\end{claim}
	\begin{proof}[Proof of Claim \ref{lem:K4FreeWithCobanner-cl0}]
	Let $u$ be an arbitrary vertex of $N(X) \cap N(P)$ and $w\in X$ be a neighbor of $u$. By Observation~\ref{i9}(a), $a_1$ has no neighbor in $X$. In particular, $a_1$ is not adjacent to $w$. Now if $u\in B_{23}\cup B_{24}$, then either $\{a_1,v_1,v_2,u,w\}$ or $\{v_3,v_4,u,a_1,v_1\}$ induces a $P_5$ which is a contradiction. So $u\notin B_{23}\cup B_{24}$. Now by Observation~\ref{i9}(a), $u\in A_2\cup B_{12}\cup T_2\cup T_3\cup T_4$. Since $u$ is an arbitrary vertex of $N(X)\cap N(P)$, we have $N(X) \cap N(P)\subseteq A_2\cup B_{12}\cup T_2\cup T_3\cup T_4$. 
	\end{proof}
	
	\begin{claim}\label{lem:K4FreeWithCobanner-cl0.1}
If $x\in X$ and $N(x) \cap T_2= \emptyset$, then the robber gets captured.
\end{claim}
\begin{proof}[Proof of Claim~\ref{lem:K4FreeWithCobanner-cl0.1}]
Suppose that $x\in X$ and $N(x) \cap T_2= \emptyset$. By Observation~\ref{i9}(c), $x$ is at distance~$2$ from $P$ and hence $x$ has a neighbor in $N(P)$, say $y$. Since $N(x)\cap T_2=\emptyset$, we have $y\notin T_2$; thus $y$ is adjacent to $v_2$ by Claim~\ref{lem:K4FreeWithCobanner-cl0}. Moreover, since $y$ is arbitrary, any neighbor of $x$ in $N(P)$ is a neighbor $v_2$. In the next turn, Cop~$1$ moves to $v_2$ and Cop~$2$ moves to $y$. So if the robber moves to a vertex of $N(P)$, then it gets captured by Cop~$1$. Hence the robber should stay in $X$. Now by Observation~\ref{i9}(b), $y$ is adjacent to every vertex of $N[x]\cap X$. So the robber gets captured by Cop~$2$. 
\end{proof}

\begin{claim}\label{lem:K4FreeWithCobanner-cl1}
If $x\in X$, $N(x) \cap T_2\neq \emptyset$, and every vertex of $N(x) \cap T_2$ has a non-neighbor in $N(x) \cap A_2$, then the robber gets captured.
\end{claim}
\begin{proof}[Proof of Claim \ref{lem:K4FreeWithCobanner-cl1}]
Suppose that $x\in X$, $N(x) \cap T_2\neq \emptyset$, and every vertex of $N(x) \cap T_2$ has a non-neighbor in $N(x) \cap A_2$. Recall that $a_1\in A_1$. Since $x\in X$, by Claim~\ref{lem:K4FreeWithCobanner-cl0}, $a_1$ is not adjacent to $x$. In the next turn, Cop~$1$ and Cop~$2$ move to $a_1$ and $v_1$, respectively. The robber can stay in $X$ or move to a vertex $y\in N(x)\cap N(P)$ that is not adjacent to $a_1$ and $v_1$ if exists. Suppose that such a vertex $y$ exists. Since $y\in N(X)\cap N(P)$ and $y$ is not adjacent to $v_1$, by Claim~\ref{lem:K4FreeWithCobanner-cl0}, $y\in A_2$. Now since $a_1$ is not adjacent to $y$, $\{a_1,v_1,v_2,y,x\}$ induces a $P_5$, a contradiction. So such a vertex $y$ does not exist and hence the robber should stay in $X$. By Observation~\ref{i9}(b), without loss of generality, we may assume that it stays at $x$. In the next turn, Cop~$1$ stays at $a_1$ and Cop~$2$ moves to a vertex $c\in N(x)\cap T_2$. Such a vertex $c$ exists since $N(x) \cap T_2\neq \emptyset$. To avoid immediate capture, the robber should move to a vertex $r \in X\cup N(P)$ that is not adjacent to $a_1$ and $c$. By Observation~\ref{i9}(b), $c$ is adjacent to every vertex of $N[x]\cap X$. So $r\notin X$ and hence $r \in N(P)$. Since $r$ is adjacent to $x$, $r\in N(X)\cap N(P)$. Then by Claim~\ref{lem:K4FreeWithCobanner-cl0}, $r\in A_2\cup B_{12}\cup T_2\cup T_3\cup T_4$. Now we show that $a_1$ is adjacent to every vertex of $N(x)\cap T_2$. For the sake of contradiction, suppose that $w\in N(x)\cap T_2$ is not adjacent to $a_1$. By our assumption in the claim, $w$ has a non-neighbor in $N(x)\cap A_2$, say $a_2$. Recall that $a_1$ is not adjacent to $x$. Then either $\{a_1,a_2,x,w,v_3\}$ or $\{a_1,v_1,v_2,a_2,x\}$ induces a $P_5$, a contradiction. So $a_1$ is adjacent to every vertex of $N(x)\cap T_2$. In particular, $a_1$ is adjacent to $c$. Now since $a_1$ is adjacent to every vertex of $N(x)\cap T_2$ and not adjacent to $r$, we have $r\notin T_2$. So $r\in A_2\cup B_{12}\cup T_3\cup T_4$. Then $\{v_2,r,x,c,a_1\}$ induces a $P_5$, a contradiction. So such a vertex $r$ does not exist. Hence the robber cannot escape from $x$ and gets captured by Cop~$2$.
\end{proof}

\begin{claim}\label{lem:K4FreeWithCobanner-cl2}
If $x\in X$ and there exists a vertex in $N(x) \cap T_2$ that is adjacent to every vertex of $N(x) \cap A_2$, then the robber gets captured.
\end{claim}
\begin{proof}[Proof of Claim \ref{lem:K4FreeWithCobanner-cl2}]
Let $x\in X$ and $c$ be a vertex of $N(x)\cap T_2$ that is adjacent to every vertex of $N(x) \cap A_2$. In the next turn, Cop~$1$ moves to $c$ and Cop~$2$ moves to $v_1$. To avoid immediate capture, the robber should move to a vertex $r$ that is not adjacent to $c$ and $v_1$. By Observation~\ref{i9}(b), $c$ is adjacent to every vertex of $N[x]\cap X$. So $r\notin X$ and hence $r\in N(x)\cap N(P)$. Since $x\in X$, we have $r\in N(X)\cap N(P)$. Now by Claim~\ref{lem:K4FreeWithCobanner-cl0}, $r\in A_2\cup B_{12}\cup T_2\cup T_3\cup T_4$. Again since $r$ is not adjacent to $v_1$, we have $r\in A_2$. This is a contradiction since $c$ is adjacent to every vertex of $N(x)\cap A_2$. So such a vertex $r$ does not exist. Hence the robber cannot escape from $x$ and gets captured by Cop~$1$.
\end{proof}

\begin{claim}\label{lem:K4FreeWithCobanner-cl3}
If $x\in B_{34}$ and $B_{13}\neq\emptyset$ or $B_{14}\neq \emptyset$, then the robber gets captured.
\end{claim}	
\begin{proof}[Proof of Claim~\ref{lem:K4FreeWithCobanner-cl3}]
 Suppose that $x\in B_{34}$ and $B_{13}\neq\emptyset$. Let $b$ be any vertex of $B_{13}$. Since $\{b,v_1,v_2,v_4,x\}$ does not induce a $P_5$, $b$ is adjacent to $x$.	Again, $b$ is adjacent to $a_1$; otherwise $\{v_4,v_3,b,v_1,a_1\}$ induces a $P_5$. First we show that $B_{14}=\emptyset$. For the sake of contradiction, let $b'$ be any vertex of $B_{14}$. Note that $b'$ is a neighbor of $a_1$; otherwise $\{v_3,v_4,b',v_1,a_1\}$ induces a $P_5$. If $b$ is not adjacent to $b'$, then $\{b,a_1,b',v_4,v_2\}$ induces a $P_5$, a contradiction. So $b$ is adjacent to $b'$. Again since $\{a_1,v_1,v_2,v_3,x\}$ does not induce a $P_5$, $a_1$ is adjacent to $x$. Then either $\{a_1,b,b',x\}$ induces a $K_4$ or $\{b', v_1,v_2,v_3,x\}$ induces a $P_5$ which is a contradiction. Thus $B_{14}=\emptyset$. 

In the next turn, Cop~$1$ and Cop~$2$ moves to $b$ and $v_1$, respectively. Recall that $b$ is adjacent to $x$. To avoid immediate capture, the robber must move to a vertex $r$ that is not adjacent to $v_1$ and $b$. By Observation~\ref{i8}, $r\notin T_1$. Since $x\in B_{34}$ and $r$ is a neighbor of $x$, by Claim~\ref{lem:K4FreeWithCobanner-cl0}, $r\notin X$. Recall that $A_3=\emptyset$ and $A_4=\emptyset$ and hence $r\notin A_3\cup A_4$. Again since $r$ is not adjacent to $v_1$ and $b$, we have $r\in \{v_4\}\cup A_2\cup B_{23}\cup B_{24}\cup B_{34}$. Since neither $\{r,v_4,v_3,x\}$ induces a $K_4$ nor $\{v_1,b,v_3,v_4,r\}$ induces a $P_5$, we have $r\notin B_{34} \cup B_{24}$. Recall that $b$ is adjacent to $a_1$. If $r\in A_2$, then $\{a_1,b,v_3,v_2,r\}$ or $\{b,a_1,r,v_2,v_4\}$ induces a $P_5$, a contradiction. So $r \notin A_2$ and hence $r\in \{v_4\}\cup B_{23}$. In the next turn, Cop~$1$ moves to $v_3$ and Cop~$2$ moves to $v_2$. If $r=v_4$, then, since $A_4=B_{14}=\emptyset$, we have $N[r]\subseteq N[v_2]\cup N[v_3]$ implying that the robber gets captured. So we may assume that $r \in B_{23}$. Note that $r$ does not have any neighbor in $A_1$; otherwise for any neighbor $a\in A_1$ of $r$, $\{v_1,a,r,v_3,v_4\}$ induces a $P_5$. Since $A_4=B_{14}=\emptyset$, to avoid immediate capture, the robber should move to a vertex $r' \in X$. Since $b\in B_{13}$, by Claim~\ref{lem:K4FreeWithCobanner-cl0}, $r'$ is not adjacent to $b$. Then $\{r',r,v_2,v_1,b\}$ induces a $P_5$ which is a contradiction. So such a vertex $r'$ does not exist. Hence the robber cannot escape from $r$ and gets captured. Similarly, we can show that the robber gets captured if $B_{14}\neq \emptyset$ (due to symmetry). 
\end{proof}	

\begin{claim}\label{lem:K4FreeWithCobanner-cl4}
If $x\in B_{34}$, $B_{13}=B_{14}=\emptyset$, and every vertex in $N(x) \cap A_2$ has a non-neighbor in $A_1$, then the robber gets captured.
\end{claim}
\begin{proof}[Proof of Claim~\ref{lem:K4FreeWithCobanner-cl4}]
Suppose that $x\in B_{34}$, $B_{13}=B_{14}=\emptyset$, and every vertex in $N(x) \cap A_2$ has a non-neighbor in $A_1$. In the next turn, Cop~$1$ stays at $v_1$ and Cop~$2$ moves to $v_3$. To avoid immediate capture, the robber must move to a vertex $r$ that is not adjacent to $v_1$ and $v_3$. Since $x\in B_{34}$, by Claim~\ref{lem:K4FreeWithCobanner-cl0}, $r\notin X$. Again since $r$ is not adjacent to $v_1$ and $v_3$, we have $r\in A_2\cup A_4\cup B_{24}$. Recall that $A_3=A_4=\emptyset$. So $r\notin A_4$ and hence $r\in A_2\cup B_{24}$. 

First assume that $r\in B_{24}$. In the next turn, Cop~$1$ moves to $v_2$ and Cop~$2$ stays at $v_3$. To avoid immediate capture, the robber must move to a vertex $r'$ that is not adjacent to $v_2$ and $v_3$. Since $A_4 = B_{14}= \emptyset$, we have $r'\in A_1 \cup X$. Again since $\{v_3,v_4,r,r',v_1\}$ does not induce a $P_5$, $r'\notin A_1$. So $r'\in X$. By Claim~\ref{lem:K4FreeWithCobanner-cl0}, $a_1$ is not adjacent to $r'$. Note that $a_1$ is not adjacent to $r$; otherwise $\{v_3,v_4,r,a_1,v_1\}$ induces a $P_5$. Then $\{a_1,v_1,v_2,r,r'\}$ induces a $P_5$ which is a contradiction. So such a vertex $r'$ does not exist. Hence the robber cannot escape from $r$ and gets captured by Cop~$1$. 
	
	Now suppose that $r\in A_2$. Since $r\in N(x)\cap A_2$, due to our assumption in the claim, $r$ has a non-neighbor in $A_1$, say $a'$. Recall that Cop~$1$, Cop~$2$, and the robber are at $v_1,v_3,$ and $r$, respectively. In the next turn, Cop~$1$ moves to $v_2$ and Cop~$2$ stays at $v_3$. To avoid immediate capture, the robber must move to a vertex $r'$ that is not adjacent to $v_2$ and $v_3$. Since $A_4=B_{14}=\emptyset$, we have $r'\in A_1\cup X$. Suppose that $r'\in X$. Note that by Claim~\ref{lem:K4FreeWithCobanner-cl0}, $a'$ is not adjacent to $r'$. Then $\{a',v_1,v_2,r,r'\}$ induces a $P_5$ which is a contradiction. So we may assume that $r'\in A_1$. In the next turn, Cop~$1$ and Cop~$2$ move to $r$ and $v_2$, respectively. To avoid immediate capture, robber must move to a vertex $r''$ that is not adjacent to $r$ and $v_2$. Since $r'\in A_1$, by Claim~\ref{lem:K4FreeWithCobanner-cl0}, $r''\notin X$. Again since $A_3=A_4=B_{13}=B_{14}=\emptyset$ and $r''$ is not adjacent to $v_2$, we have $r''\in A_1\cup B_{34}\cup T_2$. Note that $r''\notin A_1$; otherwise $\{r'',r',r,v_2,v_3\}$ induces a $P_5$. Recall that $a'\in A_1$ is a non-neighbor of $r$. Since neither $\{a',v_1,v_2,v_3,r''\}$ nor $\{r,v_2,v_1,a',r''\}$ induces a $P_5$, we have $r''\notin B_{34}$. So $r''\in T_2$. Then either $\{r'',x,v_3,v_4\}$ induces a $K_4$ or $\{r,x,v_3,r'',v_1\}$ induces a $P_5$ which is a contradiction. So such a vertex $r''$ does not exist. Hence the robber cannot escape from $r'$ and gets captured by Cop~$1$.
\end{proof}

We now return to the proof of Lemma~\ref{lem:K4FreeWithCobanner}. Recall that Cop~$1$, Cop~$2$, and the robber are at $v_1$, $v_2$, and $x\in B_{34}\cup X$, respectively. First assume that $x\in X$. If $N(x) \cap T_2= \emptyset$, then by Claim~\ref{lem:K4FreeWithCobanner-cl0.1}, the robber gets captured. So we may assume that $N(x) \cap T_2\neq \emptyset$. If every vertex of $N(x) \cap T_2$ has a non-neighbor in $N(x) \cap A_2$, then by Claim~\ref{lem:K4FreeWithCobanner-cl1}, the robber gets captured. So we may assume that there exists a vertex in $N(x) \cap T_2$ that is adjacent to every vertex of $N(x) \cap A_2$. Then by Claim~\ref{lem:K4FreeWithCobanner-cl2}, the robber gets captured.

	Now assume that $x\in B_{34}$. If $B_{13}\neq\emptyset$ or $B_{14}\neq\emptyset$, then by Claim~\ref{lem:K4FreeWithCobanner-cl3}, the robber gets captured. So assume that $B_{13}=B_{14}=\emptyset$. Now if every vertex of $N(x) \cap A_2$ has a non-neighbor in $A_1$, then by Claim~\ref{lem:K4FreeWithCobanner-cl4}, the robber gets captured. Note that Claim~\ref{lem:K4FreeWithCobanner-cl4} also includes the case $N(x)\cap A_2=\emptyset$. So we may assume that there exists a vertex $a_2\in N(x)\cap A_2$ such that $a_2$ is adjacent to every vertex of $A_1$. Recall that the cops are at $v_1$ and $v_2$ whereas the robber is at $x\in B_{34}$. In the next turn, Cop~$1$ and Cop~$2$ move to $v_2$ and $a_2$, respectively. Now if the robber does not move, then, since $a_2$ is adjacent to $x$, it gets captured by Cop~$2$. To avoid immediate capture, it must move to a vertex $r$ that is not adjacent to $v_2$ and $a_2$. Since $x\in B_{34}$, by Claim~\ref{lem:K4FreeWithCobanner-cl0}, $x$ does not have any neighbor in $X$ implying that $r\notin X$. Recall that $A_3=A_4=B_{13}=B_{14}=\emptyset$ and $a_2$ is adjacent to every vertex of $A_1$. So $r\in B_{34}\cup T_2$. Then $\{x,r,v_3,v_4\}$ induces a $K_4$ which is a contradiction. So such a vertex $r$ does not exist. Hence the robber cannot escape from $x$ and gets captured implying that $cop(G)\leq 2$. This completes the proof of Lemma~\ref{lem:K4FreeWithCobanner}.
\end{proof}

	In the following lemma, we show that the cop number of any connected ($P_5,K_4$, co-banner)-free graph that has an induced butterfly is at most~$2$. The idea of the proof is similar to the proof of Lemma~\ref{lem:K4FreeWithCobanner}.

\begin{lemma}\label{lem:K4FreeWithButterfly}
	Let $G$ be a connected $(P_5,K_4$, co-banner$)$-free graph that has an induced butterfly. Then $cop(G) \leq 2$.
\end{lemma}

\begin{proof}
Let $\{b,v_1,v_2,v_3,v_4\}$ induce a butterfly in $G$ with edge set $\{bv_1,v_1v_2,bv_2,v_2v_3,v_3v_4,v_4v_2\}$. Note that $P=\{v_1,v_2,v_3,v_4\}$ induces a paw in $G$. Define the sets $A_i,B_{ij},T_i,D$, and $X$ around $P$ for every $1 \leq i,j \leq 4$ and $i<j$ as defined in \ref{pawstruct1}. Note that $b \in B_{12}$. Again note that $A_1=\emptyset$; otherwise $G[P\cup A_1]$ contains an induced co-banner. If $B_{14} \neq \emptyset$, then for any $b' \in B_{14}$, $\{b,v_1,b',v_4,v_3\}$ induces a $P_5$ or a co-banner, a contradiction. So $B_{14}=\emptyset$. In the first turn, we place Cop~$1$ and Cop~$2$ at $v_3$ and $v_4$, respectively. To avoid immediate capture, the robber must choose a vertex $r$ that is not adjacent to $v_3$ and $v_4$. Since $A_1=\emptyset$, we have $r\in \{v_1\} \cup A_2 \cup B_{12} \cup X$. 
 
 \begin{claim}\label{lem:K4FreeWithButterfly-cl1}
 If a vertex $w$ has a neighbor in $\{v_1\}\cup A_2\cup B_{12}$, then $w$ is adjacent to $v_2$ or $v_3$.
 \end{claim}
 \begin{proof}[Proof of Claim~\ref{lem:K4FreeWithButterfly-cl1}]
 Let $w$ be a vertex that is not adjacent to $v_2$ and $v_3$. Note that $w$ is not a vertex of the butterfly induced by $\{b,v_1,v_2,v_3,v_4\}$. To prove the claim, it is sufficient to show that $w$ does not have any neighbor in $\{v_1\}\cup A_2\cup B_{12}$. For the sake of contradiction, suppose that $w$ has a neighbor $u$ in $\{v_1\}\cup A_2\cup B_{12}$. Note that $w$ is adjacent to $v_4$; otherwise $\{w,u,v_2,v_3,v_4\}$ induces a co-banner. If $w$ is adjacent to $v_1$, then $w\in B_{14}$, a contradiction to the fact that $B_{14}=\emptyset$. So $w$ is not adjacent to $v_1$. Recall that $b\in B_{12}$. Note that $w$ is not adjacent to $b$; otherwise $\{v_1,b,w,v_4,v_3\}$ induces a $P_5$. Now $\{v_1,b,v_2,v_4,w\}$ induces a co-banner, a contradiction. Therefore, $w$ does not have any neighbor in $\{v_1\}\cup A_2\cup B_{12}$.
 \end{proof}
 
 \begin{claim}\label{lem:K4FreeWithButterfly-cl2}
 $N(P)\cap N(X) \subseteq N(b) \cap (B_{23}\cup B_{24}\cup T_2)$.
 \end{claim}
 \begin{proof}[Proof of Claim~\ref{lem:K4FreeWithButterfly-cl2}]
 Let $x$ be an arbitrary vertex of $X$. By Observation~\ref{i9}(c), $x$ is at distance $2$ from $P$ and hence $N(x) \cap N(P)\neq \emptyset$. Let $u\in N(x) \cap N(P)$. Since $\{v_3,v_4,v_2,u,x\}$ does not induce a co-banner, $u \notin A_2\cup B_{12}$. Again since neither $\{b,u,v_1,v_2\}$ induces a $K_4$ nor $\{v_4,v_3,u,v_1,b\}$ induces a $P_5$, we have $u \notin T_3 \cup T_4$. Hence by Observation~\ref{i9}(a), $u \in B_{23}\cup B_{24}\cup T_2$ implying that $N(x) \cap N(P) \subseteq B_{23}\cup B_{24}\cup T_2$. In particular, $b$ is not adjacent to $x$. Since $\{b,v_1,v_2,u,x\}$ does not induce a co-banner, $b$ is a neighbor of $u$. So we have $N(x)\cap N(P)\subseteq N(b)$ and hence $N(x)\cap N(P)\subseteq N(b) \cap (B_{23}\cup B_{24}\cup T_2)$. Since $x$ is an arbitrary vertex of $X$, we have $N(P)\cap N(X) \subseteq N(b) \cap (B_{23}\cup B_{24}\cup T_2)$.
 \end{proof}
	
 We now return to the proof of Lemma~\ref{lem:K4FreeWithButterfly}. First suppose that $r\in  \{v_1\} \cup A_2 \cup B_{12}$. In the next turn, Cop~$1$ stays at $v_3$ and Cop~$2$ moves to $v_2$. Now by Claim~\ref{lem:K4FreeWithButterfly-cl1}, every vertex of $N[r]$ is adjacent to $v_2$ or $v_3$. So wherever the robber moves, it gets captured. Hence we may assume that $r\in X$. By Observation~\ref{i9}(c) and Claim~\ref{lem:K4FreeWithButterfly-cl2}, there exists a vertex $u \in N(P)$ that is a common neighbor of $r$ and $v_i$ for some $i \in \{3,4\}$. In the next turn, Cop~$1$ moves to $v_2$ and Cop~$2$ chooses $v_i$ as its position. By Claim~\ref{lem:K4FreeWithButterfly-cl2}, every vertex of $N(P) \cap N(r)$ is adjacent to $v_2$ or $v_i$. So if the robber moves to $N(P)$, then it gets captured and hence the robber should stay in $X$. By Observation~\ref{i9}(b), without loss of generality, we may assume that the robber stays at $r$. In the next turn, Cop~$1$ and Cop~$2$ move to $b$ and $u$, respectively. To avoid immediate capture, the robber must move to a vertex of $X\cup N(P)$. By Observation~\ref{i9}(b), $u$ is adjacent to every vertex of $N[r]\cap X$. So if the robber stays in $X$, then it gets captured by Cop~$2$. Hence the robber must move to a vertex of $N(P)$. Now by Claim~\ref{lem:K4FreeWithButterfly-cl2}, $b$ is adjacent to every vertex of $N(X)\cap N(P)$. So the robber gets captured by the cop at $b$. Therefore, $cop(G)\leq 2$. This completes the proof of Lemma~\ref{lem:K4FreeWithButterfly}.
 \end{proof}

In the following lemma, we prove that the robber gets captured in a connected $(P_5,K_4,$ co-banner, butterfly$)$-free graph if at the end of a round, the two cops and the robber are at some specific vertices of the graph.

\begin{lemma}\label{lem:K4FreePositionOfCop}
	Suppose that $G$ is a connected $(P_5,K_4$, co-banner, butterfly$)$-free graph and $P$ induces a paw in $G$. If at the end of a round in the game of cops and robber, two cops are at two distinct degree~$2$ vertices of the graph $G[P]$ and the robber is at a vertex at distance~$2$ from $P$, then the robber gets captured.
\end{lemma}
\begin{proof}
Let $P= \{v_1,v_2,v_3,v_4\}$ and the edge set of the paw induced by $P$ be $\{v_1v_2,v_2v_3,v_3v_4,v_4v_2\}$. Define the sets $A_i,B_{ij},T_i,D$, and $X$ around $P$ as defined in \ref{pawstruct1} for every $1 \leq i,j \leq 4$ and $i<j$. We may assume that Cop~$1$ is at $v_3$, Cop~$2$ is at $v_4$, and the robber is at $r\in X$ at the end of a round in the game of cops and robber. Note that it is now cops' turn to move. Again note that $A_1 = B_{12}=\emptyset$; otherwise $G[P\cup A_1]$ or $G[P\cup B_{12}]$ contains an induced co-banner or butterfly. To proceed further, we prove a series of claims.
	
	\begin{claim}\label{i10} 
	$N(X) \cap N(P) \subseteq B_{23}\cup B_{24}\cup T_2\cup T_3\cup T_4$.
	\end{claim}
	\begin{proof}[Proof of Claim~\ref{i10}]
	 Let $y$ be an arbitrary vertex of $N(X) \cap N(P)$ and $x$ be a neighbor of $y$ in $X$. Since $B_{12} = \emptyset$, by Observation~\ref{i9}(a), $y \in A_2\cup B_{23}\cup B_{24}\cup T_2\cup T_3\cup T_4$. Again since $\{v_3,v_4,v_2,y,x\}$ does not induce a co-banner, we have $y \notin A_2$. So $y\in B_{23}\cup B_{24}\cup T_2\cup T_3\cup T_4$. 
	\end{proof}

\begin{claim}\label{lem:K4FreePositionOfCop-cl1}
If $N(r)\cap B_{23}\neq \emptyset$, $T_3\neq \emptyset$, and $T_4\neq \emptyset$, then the robber gets captured.
\end{claim}
\begin{proof}[Proof of Claim~\ref{lem:K4FreePositionOfCop-cl1}]
Let $b \in N(r)\cap B_{23}$, $c\in T_3$, and $c'\in T_4$. In the next turn, Cop~$1$ and Cop~$2$ move to $b$ and $v_3$, respectively. Since $b$ is adjacent to $r\in X$, by Observation~\ref{i9}(b), $b$ is adjacent to every vertex of $N[r] \cap X$. So if the robber stays in $X$, then it gets captured by Cop~$1$. So to avoid immediate capture, the robber must move to a vertex $r'\in N(r)\cap N(P)$ that is not adjacent to $b$ and $v_3$. Since $r'\in N(X)\cap N(P)$ and $r'$ is not adjacent to $v_3$, by Claim~\ref{i10}, $r' \in  B_{24}\cup T_3$. Suppose that $r' \in B_{24}$. Since $\{b,v_3,v_2,c'\}$ does not induce a $K_4$, $b$ is not adjacent to $c'$. Again since $\{r',v_4,v_2,c\}$ and $\{c,v_1,v_2,c'\}$ do not induce any $K_4$, $c$ is not adjacent to $r'$ and $c'$. Note that $r'$ is adjacent to $c'$; otherwise $\{r',v_4,v_3,c',v_1\}$ induces a $P_5$. Similarly, $b$ is adjacent to $c$; otherwise $\{b,v_3,v_4,c,v_1\}$ induces a $P_5$. Then $\{b,c,v_1,c',r'\}$ induces a $P_5$ which is a contradiction to our assumption that $r' \in B_{24}$. So we have $r'\in T_3$. Now since $r'$ is not adjacent to $b$, $\{b,v_3,v_4,r',v_1\}$ induces a $P_5$, a contradiction. So such a vertex $r'$ does not exist. Hence the robber cannot escape from $r$ and gets captured.
\end{proof}
	
	\begin{claim}\label{lem:K4FreePositionOfCop-cl2}
If $T_4=\emptyset$ and $N(r)\cap T_3\neq \emptyset$, then the robber gets captured.
\end{claim}
\begin{proof}[Proof of Claim~\ref{lem:K4FreePositionOfCop-cl2}]
	 Let $T_4=\emptyset$ and $c\in N(r)\cap T_3$. In the next turn, Cop~$1$ and Cop~$2$ move to $v_4$ and $c$, respectively. Since $r\in X$ and $c$ is adjacent to $r$, by Observation~\ref{i9}(b), $c$ is adjacent to every vertex of $N[r]\cap X$. So if the robber stays in $X$, then it gets captured. So we may assume that the robber moves to a vertex $r' \in N(r) \cap N(P)$ that is not adjacent to $v_4$ and $c$. Now since $T_4 = \emptyset$ and $r'$ is not adjacent to $v_4$, by Claim~\ref{i10}, $r'\in B_{23}$. Then $\{v_1,c,v_4,v_3,r'\}$ induces a $P_5$, a contradiction. So such a vertex $r'$ does not exist. Hence the robber cannot escape from $r$ and gets captured. 
\end{proof}

\begin{claim}\label{lem:K4FreePositionOfCop-cl3}
If $N(r)\cap (T_3\cup T_4)=\emptyset$, $N(r)\cap B_{23}\neq \emptyset$, $N(r)\cap B_{24}\neq\emptyset$, and $B_{13}\neq\emptyset$ or $B_{14}\neq\emptyset$, then the robber gets captured.
\end{claim}
\begin{proof}[Proof of Claim~\ref{lem:K4FreePositionOfCop-cl3}]
	Let $N(r)\cap (T_3\cup T_4)=\emptyset$ and $b \in N(r)\cap B_{23}$. First assume that $B_{13}\neq \emptyset$. Let $b_1\in B_{13}$. In the next turn, Cop~$1$ stays at $v_3$ and Cop~$2$ moves to $v_2$. Note that by Claim~\ref{i10}, every vertex in $N(r) \cap N(P)$ is adjacent to $v_2$ or $v_3$. So if the robber moves to a vertex of $N(P)$, then it gets captured. Hence the robber should stay in $X$. By Observation~\ref{i9}(b), without loss of generality, we may assume that the robber stays at $r$. In the next turn, Cop~$1$ and Cop~$2$ move to $b$ and $b_1$, respectively. Since $b$ is adjacent to $r\in X$, by Observation~\ref{i9}(b), $b$ is adjacent to every vertex of $N[r]\cap X$. So if the robber stays in $X$, then its gets captured by Cop~$1$. To avoid immediate capture, robber must move to a vertex $r'\in N(r)\cap N(P)$ that is not adjacent to $b$ and $b_1$. Since $r'\in N(X)\cap N(P)$, by Claim~\ref{i10}, $r' \in B_{23}\cup B_{24}\cup T_2\cup T_3\cup T_4$. Since $b_1\in B_{13}$ and $r\in X$, by Claim~\ref{i10}, $b_1$ is not adjacent to $r$. Now since $\{b_1,v_1,v_2,r',r\}$ does not induce a $P_5$, we have $r' \notin B_{23}\cup B_{24}$. So $r' \in T_2\cup T_3\cup T_4$. Again since $N(r)\cap (T_3\cup T_4)=\emptyset$, we have $r' \in T_2$. Note that $b$ is adjacent to $b_1$; otherwise $\{b_1,v_1,v_2,b,r\}$ induces a $P_5$. Then $\{b_1,b,r,r',v_4\}$ induces a $P_5$ which is a contradiction. So such a vertex $r'$ does not exist. Hence the robber cannot escape from $r$ and gets captured. Now if $B_{14}\neq \emptyset$, then, since $N(r)\cap B_{24}\neq \emptyset$, we may conclude that the robber gets captured (due to symmetry). 
\end{proof}
	
\begin{claim}\label{lem:K4FreePositionOfCop-cl4}
If $T_4=N(r)\cap T_3=B_{14}=\emptyset$ and $N(r)\cap B_{23}\neq \emptyset$, then the robber gets captured.
\end{claim}
\begin{proof}[Proof of Claim~\ref{lem:K4FreePositionOfCop-cl4}]
Let $T_4=N(r)\cap T_3=B_{14}=\emptyset$ and $b \in N(r)\cap B_{23}$. First assume that $r$ has no neighbor in $T_2$. In the next turn, Cop~$1$ and Cop~$2$ move to $b$ and $v_2$, respectively. Since $b$ is adjacent to $r\in X$, by Observation~\ref{i9}(b), $b$ is adjacent to every vertex of $N[r]\cap X$. So if the robber stays in $X$, then it gets captured by Cop~$1$. So the robber must move to a vertex of $N(r) \cap N(P)$. Now since $r\in X$ and $r$ has no neighbor in $T_2$, by Claim~\ref{i10}, the robber must move to a vertex of $B_{23}\cup B_{24}\cup T_3\cup T_4$. Then the cop at $v_2$ captures the robber. 
	
	So we may assume that $r$ has a neighbor in $T_2$, say $c_1$. In the next turn, Cop~$1$ moves to $c_1$ and Cop~$2$ stays at $v_4$. Since $c_1$ is adjacent to $r\in X$, by Observation~\ref{i9}(b), $c_1$ is adjacent to every vertex of $N[r]\cap X$. So if the robber stays in $X$, then it gets captured by Cop~$1$. So to avoid immediate capture, the robber must move to vertex $r'\in N(r)\cap N(P)$ that is not adjacent to $c_1$ and $v_4$. Since $T_4=\emptyset$ and $r'$ is not adjacent to $v_4$, by Claim~\ref{i10}, we have $r'\in B_{23}$. In the next turn, Cop~$1$ stays at $c_1$ and Cop~$2$ moves to $v_3$. Again to avoid immediate capture, the robber must move to a vertex $r''$ that is not adjacent to $v_3$ and $c_1$. Since $A_1=B_{12}=B_{14}=\emptyset$ and $r''$ is not adjacent to $v_3$, we have $r''\in A_2\cup A_4\cup B_{24}\cup T_3\cup X$. If $r''\in X$, then $\{r'',r',v_2,v_4,c_1\}$ induces a $P_5$, a contradiction. Thus $r''\notin X$. Since $\{r'',v_2,r',v_1,c_1\}$ does not induce a co-banner, we have $r'' \notin A_2\cup B_{24}$. If $r''\in T_3$, then by our assumption in the claim that $N(r)\cap T_3=\emptyset$, we conclude that $r$ is not adjacent to $r''$. Now $\{r'',v_2,v_4,c_1,r\}$ induces a co-banner, a contradiction. So $r''\notin T_3$. Hence $r''\in A_4$. Since $r\in X$ and $r''\in A_4$, by Claim~\ref{i10}, $r''$ is not adjacent to $r$. Then $\{r'',r',r,c_1,v_1\}$ induces a $P_5$ which is a contradiction. So such a vertex $r''$ does not exist. Hence the robber cannot escape from $r'$ and gets captured.
\end{proof}

\begin{claim}\label{lem:K4FreePositionOfCop-cl5}
If $N(r)\cap B_{23}=N(r)\cap B_{24}=\emptyset$, then the robber gets captured.
\end{claim}
\begin{proof}[Proof of Claim~\ref{lem:K4FreePositionOfCop-cl5}]
Suppose that $N(r)\cap B_{23}=N(r)\cap B_{24}=\emptyset$. So by Claim~\ref{i10}, $N(r) \cap N(P) \subseteq T_2\cup T_3\cup T_4$. In the next turn, Cop~$1$ moves to $v_2$ and Cop~$2$ stays at $v_4$. Note that every vertex of $N(r) \cap N(P)$ is adjacent to $v_2$ or $v_4$. So to avoid capture, the robber should stay in $X$. By Observation~\ref{i9}(b), without loss of generality, we may assume that the robber stays at $r$. In the next turn, Cop~$1$ and Cop~$2$ move to $v_1$ and $v_2$, respectively. Since $N(r) \cap N(P)\subseteq T_2\cup T_3\cup T_4$, every vertex of $N(r) \cap N(P)$ is adjacent to $v_1$. So if the robber moves to a vertex of $N(P)$, then it gets captured by Cop~$1$ and hence the robber should stay in $X$. By Observation~\ref{i9}(b), without loss of generality, we may assume that the robber stays at $r$. Since $r$ is at distance $2$ from $P$, $N(r)\cap N(P)\neq \emptyset$. Let $x\in N(r)\cap N(P)$. In the next turn, Cop~$1$ and Cop~$2$ move to $x$ and $v_1$, respectively. Since $x\in N(X)\cap N(P)$, by Observation~\ref{i9}(b), $x$ is adjacent to every vertex of $N[r]\cap X$. So if the robber stays in $X$, then Cop~$1$ captures it and hence the robber should move to a vertex of $N(r)\cap N(P)$. Then the robber gets captured by the cop at $v_1$ since every vertex of $N(r)\cap N(P)$ is adjacent to $v_1$.
\end{proof}

	We now return to the proof of Lemma~\ref{lem:K4FreePositionOfCop}. First assume that $N(r)\cap B_{23}\neq \emptyset$ and $N(r)\cap B_{24}\neq \emptyset$. If $T_3\neq\emptyset$ and $T_4\neq\emptyset$, then by Claim~\ref{lem:K4FreePositionOfCop-cl1}, the robber gets captured. So we may assume that $T_3=\emptyset$ or $T_4=\emptyset$. Due to symmetry, we may assume that $T_4=\emptyset$. Now if $N(r)\cap T_3\neq \emptyset$, then by Claim~\ref{lem:K4FreePositionOfCop-cl2}, the robber gets captured. So we may further assume that $N(r)\cap T_3=\emptyset$. Now if $B_{13}\neq\emptyset$ or $B_{14}\neq\emptyset$, then by Claim~\ref{lem:K4FreePositionOfCop-cl3}, the robber gets captured. So we may assume that $B_{13}=B_{14}=\emptyset$. Now by Claim~\ref{lem:K4FreePositionOfCop-cl4}, the robber gets captured. This implies that the robber gets captured if $N(r)\cap B_{23}\neq \emptyset$ and $N(r)\cap B_{24}\neq \emptyset$. So we may assume that $N(r)\cap B_{23}=\emptyset$ or $N(r)\cap B_{24}=\emptyset$. If $N(r)\cap B_{23}=\emptyset$ and $N(r)\cap B_{24}=\emptyset$, then by Claim~\ref{lem:K4FreePositionOfCop-cl5}, the robber gets captured. So we may assume that exactly one of the sets $N(r)\cap B_{23}$ and $N(r)\cap B_{24}$ is empty. Due to symmetry, we may further assume that $N(r)\cap B_{23}=\emptyset$ and $N(r)\cap B_{24}\neq\emptyset$. Let $b\in N(r)\cap B_{24}$. Recall that Cop~$1$, Cop~$2$, and the robber are at $v_3,v_4,$ and $r\in X$, respectively. In the next turn, Cop~$1$ and Cop~$2$ move to $v_4$ and $b$, respectively. Since $b$ is adjacent to $r\in X$, by Observation~\ref{i9}(b), $b$ is adjacent to every vertex in $N[r] \cap X$. So if the robber stays in $X$, then it gets captured by Cop~$2$. To avoid immediate capture, the robber must move to a vertex $r'\in N(r)\cap N(P)$ that is not adjacent to $b$ and $v_4$. Since $N(r)\cap B_{23}=\emptyset$, $r'\notin B_{23}$. Again since $r'$ is adjacent to $r\in X$ and is not adjacent to $v_4$, by Claim~\ref{i10}, $r'\in T_4$. Then $\{b,v_4,v_3,r',v_1\}$ induces a $P_5$ which is a contradiction. So such a vertex $r'$ does not exist. Hence the robber cannot escape from $r$ and gets captured. This completes the proof of Lemma~\ref{lem:K4FreePositionOfCop}.
\end{proof}

In the following lemma, we show that if $G$ is a connected $(P_5,K_4$, co-banner, butterfly$)$-free graph that has an induced kite and the two cops and the robber are at some specific vertices of $G $ at the end of a round, then the robber gets captured.

\begin{lemma}\label{lem:K4FreeWithkite}
	Suppose that $G$ is a connected $(P_5,K_4$, co-banner, butterfly$)$-free graph that has an induced kite. Let $P$ induce a paw in $G$ that is contained in an induced kite of $G$. If at the end of a round in the game of cops and robber, two cops are at two distinct degree~$2$ vertices of the graph $G[P]$ and the robber is at the degree~$1$ vertex of $G[P]$, then the robber gets captured.
\end{lemma}
\begin{proof}
Let $P=\{v_1,v_2,v_3,v_4\}$ and the edge set of the paw induced by $P$ be $\{v_1v_2,v_2v_3,v_3v_4,v_4v_2\}$. We may assume that Cop~$1$, Cop~$2$, and the robber are at $v_3, v_4$, and $v_1$, respectively at the end of a round in the game of cops and robber. Note that it is now cops' turn to move. Define the sets $A_i,B_{ij},T_i,D$, and $X$ around $P$ for every $1 \leq i,j \leq 4$ and $i<j$ as defined in \ref{pawstruct1}. Since the paw induced by $P$ is contained in an induced kite of $G$, we have $B_{34}\neq \emptyset$. Let $b \in B_{34}$. Note that $A_1=B_{12}=\emptyset$; otherwise for any $u\in A_1\cup B_{12}$, $P\cup \{u\}$ induces a co-banner or a butterfly in $G$.
 
\begin{claim}\label{lem:K4FreeWithkite-cl1}
If $A_3 \cup (B_{23} \setminus N(b))\neq\emptyset$ or $A_4 \cup (B_{24} \setminus N(b))\neq\emptyset$ or $[B_{34}, A_2]\neq\emptyset$, then the robber gets captured.
\end{claim}
\begin{proof}[Proof of Claim~\ref{lem:K4FreeWithkite-cl1}]
 Let $a\in A_3 \cup (B_{23} \setminus N(b))$. First we show that every neighbor of $v_1$ is adjacent to $v_4$ or $b$. For the sake of contradiction, let $x$ be a neighbor of $v_1$ such that $x$ is not adjacent to $v_4$ and $b$. Clearly, $x\neq v_2$ and hence $x$ is not a vertex of the kite induced by $\{v_1,v_2,v_3,v_4,b\}$. Then $\{b,v_4,v_2,v_1,x\}$ induces a $P_5$ or a co-banner, a contradiction. So every neighbor of $v_1$ is adjacent to $v_4$ or $b$. In the next turn, Cop~$1$ moves to $b$ and Cop~$2$ stays at $v_4$. Since every neighbor of $v_1$ is adjacent to $v_4$ or $b$, to avoid immediate capture, the robber must stay at $v_1$. Now we show that $a$ and $b$ are not adjacent. If $a \in A_3$, then, since $\{a,b,v_4,v_2,v_1\}$ does not induce a $P_5$, $a$ and $b$ are not adjacent. Again if $a\in B_{23}\setminus N(b)$, then $a$ is not adjacent to $b$. Note that $P'=\{a,v_3,b,v_4\}$ induces a paw in $G$. At the end of this round, the cops are at $b$ and $v_4$ that are degree~$2$ vertices of $G[P']$ and the robber is at $v_1$ that is at distance~$2$ from $P'$. So by Lemma~\ref{lem:K4FreePositionOfCop}, the robber gets captured. Similarly, we can show that the robber gets captured if $A_4 \cup (B_{24} \setminus N(b))\neq \emptyset$. 
	
	Now suppose that $[B_{34}, A_2]\neq \emptyset$. Let $a\in A_2$ and $b'\in B_{34}$ such that $a$ is adjacent to $b'$. Note that $P''=\{a,b',v_3,v_4\}$ induces a paw. Again note that the cops are at $v_3$ and $v_4$ that are degree~$2$ vertices of $G[P'']$ and the robber is at $v_1$ that is at distance~$2$ from $P''$. So by Lemma~\ref{lem:K4FreePositionOfCop}, the robber gets captured.
\end{proof}	
	 
	We now return to the proof of Lemma~\ref{lem:K4FreeWithkite}. If $A_3 \cup (B_{23} \setminus N(b))\neq \emptyset$ or $A_4 \cup (B_{24} \setminus N(b))\neq\emptyset$, then by Claim~\ref{lem:K4FreeWithkite-cl1}, the robber gets captured. So we may assume that $A_3=A_4 = \emptyset$ and $b$ is adjacent to every vertex of $B_{23} \cup B_{24}$. If $[B_{34}, A_2]\neq \emptyset$, then by Claim~\ref{lem:K4FreeWithkite-cl1}, the robber gets captured. So we may assume that $[B_{34}, A_2]=\emptyset$. In the next turn, Cop~$1$ stays at $v_3$ and Cop~$2$ moves to $v_2$. To avoid immediate capture, the robber must move to a vertex $r'$ that is not adjacent to $v_3$ and $v_2$. Since $A_1=\emptyset$, we have $r'\in B_{14}$. Note that $b$ is adjacent to every vertex of $B_{14}$; otherwise for any non-neighbor $y\in B_{14}$ of $b$, $\{b,v_3,v_2,v_1,y\}$ induces a $P_5$. In particular, $b$ is adjacent to $r'$. In the next turn, Cop~$1$ stays at $v_3$ and Cop~$2$ moves to $v_1$. To avoid immediate capture, the robber must move to a vertex $r''$ that is not adjacent to $v_3$ and $v_1$. Since $r'\in B_{14}$, by Observation~\ref{i9}(a), $r''\notin X$. Since $A_4 = \emptyset$, we have $r''\in A_2 \cup B_{24}$. If $r''\in B_{24}$, then, since $b$ is adjacent to every vertex of $B_{14}\cup B_{24}$, $\{b,v_4,r',r''\}$ induces a $K_4$ in $G$ which is a contradiction. So $r''\notin B_{24}$ and hence $r'' \in A_2$. In the next turn, Cop~$1$ moves to $v_2$ and Cop~$2$ stays at $v_1$. To avoid immediate capture, the robber must move to a vertex $r_1$ that is not adjacent to $v_1$ and $v_2$. Since $A_3 = A_4=\emptyset$ and $[B_{34},A_2]=\emptyset$, we have $r_1 \in X$. Then $\{v_3,v_4,v_2,r'',r_1\}$ induces a co-banner which is a contradiction. So such a vertex $r_1$ does not exist. Hence the robber cannot escape from $r''$ and gets captured. This completes the proof of Lemma~\ref{lem:K4FreeWithkite}.
\end{proof}

We now show that the cop number of any connected $(P_5,K_4$, co-banner, butterfly$)$-free graph is at most~$2$.

\begin{lemma}\label{lem:K4FreeWithoutkite}
	Let $G$ be a connected $(P_5,K_4$, co-banner, butterfly$)$-free graph. Then $cop(G) \leq 2$.
\end{lemma}
\begin{proof}
If $G$ is paw-free, then by Lemma~\ref{pawlemma}, $cop(G)\leq 2$. So we may assume that $G$ contains an induced paw, say with vertex set $P=\{v_1,v_2,v_3,v_4\}$ and edge set $\{v_1v_2, v_2v_3, v_3v_4,v_2v_4\}$. Define sets $A_{i},B_{ij}, T_i, D,$ and $X$ around $P$ for every $1 \leq i,j \leq 4$ and $i<j$ as defined in~\ref{pawstruct1}. Since $G$ is $($co-banner, butterfly$)$-free, we have $A_1=B_{12}=\emptyset$. In the first turn, we place Cop~$1$ and Cop~$2$ at $v_3$ and $v_4$, respectively. To avoid immediate capture, the robber must choose a vertex $r$ that is not adjacent to $v_3$ and $v_4$. Since $A_1=B_{12}=\emptyset$ and $r$ is not adjacent to $v_3$ and $v_4$, we have $r\in \{v_1\} \cup A_2 \cup X$. If $r\in X$, then by Lemma~\ref{lem:K4FreePositionOfCop}, the robber gets captured. So we may assume that $r\in \{v_1\}\cup A_2$. Due to symmetry, we may further assume that $r=v_1$. Now if $B_{34}\neq \emptyset$, then by Lemma~\ref{lem:K4FreeWithkite}, the robber gets captured. So assume that $B_{34}=\emptyset$.

\begin{claim}\label{lem:K4FreeWithoutkite-cl0}
If $A_3=\emptyset$ or $A_4 =\emptyset$, then the robber gets captured. 
\end{claim}
\begin{proof}[Proof of Claim~\ref{lem:K4FreeWithoutkite-cl0}]
Due to symmetry, it is sufficient to show that the robber gets captured if $A_3=\emptyset$. Let $A_3=\emptyset$. In the next turn, Cop~$1$ stays at $v_3$ and Cop~$2$ moves to $v_2$. To avoid immediate capture, the robber should move to a vertex $r'$ that is not adjacent to $v_2$ and $v_3$. Since $A_1=\emptyset$, we have $r'\in B_{14}$. In the next turn, Cop~$1$ and Cop~$2$ move to $v_2$ and $v_1$, respectively. To avoid immediate capture, the robber must move to a vertex $r''$ that is not adjacent to $v_1$ and $v_2$. Since $r'\in B_{14}$, by Observation~\ref{i9}(a), $r''\notin X$. Again since $B_{34} = A_3 = \emptyset$, we have $r'' \in A_4$. Then $\{r'',r',v_1,v_2,v_3\}$ induces a $P_5$ which is a contradiction. So such a vertex $r''$ does not exist. Hence the robber cannot escape from $r'$ and gets captured. 
\end{proof}

\begin{claim}\label{lem:K4FreeWithoutkite-cl1}
If there exists a vertex in $B_{13}$ that is adjacent to every vertex of $T_3$, then the robber gets captured. 
\end{claim}
\begin{proof}[Proof of Claim~\ref{lem:K4FreeWithoutkite-cl1}]
Let there be a vertex $b^*\in B_{13}$ such that $b^*$ is adjacent to every vertex of $T_3$. If $A_3=\emptyset$ or $A_4=\emptyset$, then by Claim~\ref{lem:K4FreeWithoutkite-cl0}, the robber gets captured. So assume that $A_3\neq \emptyset$ and $A_4\neq \emptyset$. Let $a\in A_3$ and $a'\in A_4$. Recall that Cop~$1$, Cop~$2$, and the robber are at $v_3,v_4,$ and $v_1$, respectively. In the next turn, Cop~$1$ and Cop~$2$ move to $b^*$ and $v_3$, respectively. To avoid immediate capture, the robber must move to a vertex $r'$ that is not adjacent to $b^*$ and $v_3$. Since $A_1 = B_{12}= \emptyset$, we have $r'\in B_{14}\cup T_3$. Since $b^*$ is adjacent to every vertex of $T_3$, $r'\notin T_3$. Note that this includes the case $T_3=\emptyset$. So $r'\in B_{14}$. Now since $\{a',a,v_3,v_2,v_1\}$ does not induce a $P_5$, $a$ is not adjacent to $a'$. Note that $b^*$ is not adjacent to $a$; otherwise $\{a,b^*,v_1,v_2,v_4\}$ induces a $P_5$. Again, $b^*$ is adjacent to $a'$; otherwise $\{a',v_4,v_3,b^*,v_1\}$ induces a $P_5$. Similarly, we can show that $r'$ is adjacent to $a$ and not adjacent to $a'$. Now since $r'$ is not adjacent to $b^*$, $\{a',b^*,v_1,r',a\}$ induces a $P_5$, a contradiction. So such a vertex $r'$ does not exist. Hence the robber cannot escape from $v_1$ and gets captured.
\end{proof}
	
	We now return to the proof of Lemma~\ref{lem:K4FreeWithoutkite}. If  there exists a vertex in $B_{13}$ that is adjacent to every vertex of $T_3$, then by Claim~\ref{lem:K4FreeWithoutkite-cl1}, the robber gets captured. So we may assume that $B_{13}=\emptyset$ or every vertex of $B_{13}$ has a non-neighbor in $T_3$. Recall that Cop~$1$, Cop~$2$, and the robber are at $v_3,v_4$, and $v_1$, respectively. In the next turn, Cop~$1$ moves to $v_2$ and Cop~$2$ stays at $v_4$. To avoid immediate capture, the robber must move to a vertex that is not adjacent to $v_2$ and $v_4$. Hence the robber must move to a vertex in $A_1\cup B_{13}$. Recall that $A_1=\emptyset$. Now if $B_{13}=\emptyset$, then the robber cannot escape from $v_1$ and gets captured. So assume that $B_{13}\neq \emptyset$ and the robber moves to a vertex $r_1\in B_{13}$.  Due to our assumption, $r_1$ has a non-neighbor in $T_3$, say $t$. Note that $\{t,v_2,v_4,v_3,r_1\}$ induces a kite and the cops and the robber are at the positions such that the hypothesis of Lemma~\ref{lem:K4FreeWithkite} holds. So by Lemma~\ref{lem:K4FreeWithkite}, the robber gets captured implying that $cop(G)\leq 2$. 	
\end{proof}

By using Lemma~\ref{lem:K4FreeWithCobanner}-\ref{lem:K4FreeWithoutkite}, we now show that the cop number of any connected ($P_5,K_4$)-free graph is at most~$2$.
 
\begin{theorem}\label{p5k4}
	Let $G$ be a connected $(P_5,K_4)$-free graph. Then $cop(G)\leq 2$.
\end{theorem}

\begin{proof}
If $G$ contains an induced co-banner, then by Lemma~\ref{lem:K4FreeWithCobanner}, $cop(G)\leq 2$. So we may assume that $G$ is co-banner-free. Now if $G$ contains an induced butterfly, then by Lemma~\ref{lem:K4FreeWithButterfly}, $cop(G)\leq 2$. So we may further assume that $G$ is butterfly-free. Now since $G$ is $(P_5,K_4,$ co-banner, butterfly$)$-free, by Lemma~\ref{lem:K4FreeWithoutkite}, $cop(G)\leq 2$.
	\end{proof}

\section{On the class of ($P_5,K_3 \cup K_1$)-free graphs}	\label{sec:k3Uk1} 
	Aigner and Fromme \cite{aigner} proved that for any natural number $k$, there exists a $C_3$-free graph with the cop number at least $k$. Therefore, the class of $K_3\cup K_1$-free graphs also has unbounded cop number. However, the cop number of a connected $K_3\cup K_1$-free graph that contains a $C_3$, is at most $3$ since every $C_3$ dominates the graph. In the following theorem, we show that if a connected $K_3\cup K_1$-free graph $G$ is also $P_5$-free, then $cop(G)\leq 2$.
	
	\begin{theorem}
		Let $G$ be a connected $(P_5, K_3 \cup K_1)$-free graph. Then $cop(G)\leq 2$.
	\end{theorem}
	\begin{proof}
	If $G$ is paw-free, then by Lemma~\ref{pawlemma}, $cop(G)\leq 2$. So we may assume that $G$ has an induced paw, say with vertex set $P=\{v_1,v_2,v_3,v_4\}$ and edge set $\{v_1v_2,v_2v_3,v_3v_4,v_4v_2\}$. Define the sets $A_i, B_{ij}, T_i, D$, and $X$ around $P$ as defined in~\ref{pawstruct1} for every $1 \leq i,j \leq 4$ and $i<j$. Since $G$ is $K_3 \cup K_1$-free, every $C_3$ of $G$ is a dominating cycle of $G$. So there is no vertex at distance at least~$2$ from $P$ and hence $X=\emptyset$. 
 Note that $A_1=\emptyset$; otherwise for any $a\in A_1$, $\{v_2,v_3,v_4,a\}$ induces a $K_3\cup K_1$. Moreover, $B_{34}=\emptyset$; otherwise for any $b\in B_{34}$, $\{v_1,v_3,v_4,b\}$ induces a $K_3 \cup K_1$. We divide the proof into the following two cases depending on whether $B_{12}$ is empty or not. In each case, we show that the robber gets captured by two cops after a finite number of turns.    
	
\noindent {\bf Case 1:} $B_{12}\neq \emptyset$.
		
Let $b$ be a vertex of $B_{12}$. If $A_3\neq \emptyset$, then for any $a\in A_3$, either $\{b,v_1,v_2,a\}$ induces a $K_3\cup K_1$ or $\{v_4,v_3,a,b,v_1\}$ induces a $P_5$, a contradiction. So $A_3=\emptyset$. Due to symmetry, we have $A_4=\emptyset$. Now if $B_{14}\cup T_3\neq\emptyset$, then for any $u\in B_{14}\cup T_3$, either $\{v_3,v_4,u,v_1,b\}$ induces a $P_5$ or $\{u,b,v_1,v_3\}$ induces a $K_3\cup K_1$, a contradiction. So $B_{14}\cup T_3=\emptyset$. In the first turn, we place Cop~$1$ and Cop~$2$ at $v_3$ and $v_4$, respectively. To avoid immediate capture, the robber should choose a vertex $x$ that is not adjacent to $v_3$ and $v_4$. Since $A_1=X=\emptyset$, we have $x\in A_2 \cup\{v_1\}\cup B_{12}$. In the next turn, Cop~$1$ stays at $v_3$ and Cop~$2$ moves to $v_2$. If the robber stays at $x$, then it gets captured by Cop~$2$. Since $A_1 = A_4=B_{14}=X=\emptyset$, we have $N[x]\subseteq N[v_2]\cup N[v_3]$. So wherever the robber moves, it gets captured by Cop~$1$ or Cop~$2$. 
	
\noindent {\bf Case 2:} $B_{12}= \emptyset$.
	
In the first turn, we place Cop~$1$ and Cop~$2$ at $v_1$ and $v_2$, respectively. To avoid immediate capture, the robber has to choose a vertex $x$ that is not adjacent to $v_1$ and $v_2$. Since $B_{34}=X=\emptyset$, we have $x\in A_3\cup A_4$. Due to symmetry, we may assume that $x\in A_3$. In the next turn, Cop~$1$ and Cop~$2$ move to $v_2$ and $v_4$, respectively. Note that $x$ does not have any neighbor in $B_{13}$; otherwise for any neighbor $y\in B_{13}$ of $x$, $\{x,y,v_1,v_2,v_4\}$ induces a $P_5$ in $G$. So the robber cannot move to a vertex of $B_{13}$. Moreover, since $A_1=X=\emptyset$, to avoid capture, the robber should stay in $A_3$. At the end of this round, suppose that the robber is at $x'\in A_3$. Note that $x'$ may be equal to $x$. In the next turn, Cop~$1$ moves to $v_3$ and Cop~$2$ stays at $v_4$. To avoid immediate capture, the robber must move to a vertex $r$ that is not adjacent to $v_3$ and $v_4$. Since $A_1= B_{12} =X= \emptyset $, we have $r\in A_2$. In the next turn, Cop~$1$ and Cop~$2$ move to $x'$ and $v_3$, respectively. To avoid immediate capture, the robber must move to a vertex $r'$ that is not adjacent to $x'$ and $v_3$. Since $A_1=B_{12}=X=\emptyset$, we have $r'\in A_2\cup A_4\cup B_{14}\cup B_{24}\cup T_3$. Since neither $\{r',r,x',v_3,v_4\}$ nor $\{x',v_3,v_4,r',v_1\}$ induces a $P_5$, we have $r'\notin A_2\cup B_{14}\cup T_3$. Again since $\{r',v_2,v_4,x'\}$ does not induce a $K_3\cup K_1$, we have $r'\notin B_{24}$ and hence $r'\in A_4$. In the next turn, Cop~$1$ and Cop~$2$ move to $r$ and $v_2$, respectively. To avoid immediate capture, the robber must move to a vertex $r''$ that is not adjacent to $r$ and $v_2$. Since $A_1=B_{34}=X=\emptyset$ and $r''$ is not adjacent to $v_2$, we have $r''\in A_3\cup A_4\cup B_{13} \cup B_{14} \cup T_2$. Since $\{r'',r',r,v_2,v_1\}$ does not induce a $P_5$, we have $r''\notin A_3\cup A_4$. Again since neither $\{r',r'',v_1,v_2,v_3\}$ induces a $P_5$ nor $\{r'',v_3,v_4,r\}$ induces a $K_3\cup K_1$, we have $r''\notin B_{14}\cup T_2$ and hence $r''\in B_{13}$. Note that $x'$ is not adjacent to $r''$; otherwise  $\{x',r'',v_1,v_2,v_4\}$ induces a $P_5$. Now $\{x',r,r',r'',v_1\}$ induces a $P_5$ which is a contradiction. So such a vertex $r''$ does not exist. Therefore, the robber cannot escape from $r'$ and gets captured by Cop~$1$.   
	\end{proof}

	\section{On the class of $P_3 \cup P_1$-free graphs} \label{sec:P3P1}
	
Let $G$ be a connected $P_3\cup P_1$-free graph. If $G$ has no induced $P_3$, then $G$ is isomorphic to a complete graph and hence the domination number of $G$ is~$1$. Suppose that $G$ has an induced $P_3$. Then, since $G$ is $P_3\cup P_1$-free, the set of vertices of any induced $P_3$ of $G$ is a dominating set of $G$ and hence the domination number of $G$ is at most $3$. This implies that the cop number of $G$ is at most $3$. In the following theorem, we show that two cops are sufficient to capture the robber in any connected $P_3 \cup P_1$-free graph.
	
	\begin{theorem} \label{p3up1lemma}
		Let $G$ be a connected $P_3 \cup P_1$-free graph. Then $cop(G) \leq 2$.
	\end{theorem}
	
	\begin{proof}
Let $v\in V(G)$. In the first turn, we place both Cop~$1$ and Cop~$2$ at $v$. To avoid immediate capture, the robber must choose a vertex $x$ of $G-N[v]$ if exists. We may assume that such a vertex $x$ exists; otherwise we have $V(G)=N[v]$ and hence $cop(G)=1$. Let $C$ be the component of $G-N[v]$ that contains $x$. Since $G$ is $P_3\cup P_1$-free, $G-N[v]$ is $P_3$-free and hence every component of $G-N[v]$ is a complete graph. In particular, $C$ is a complete graph. Since $G$ is connected, $N(C)\cap N(v)\neq \emptyset$. Let $r\in N(C)\cap N(v)$. Note that while one of the cops stays at $v$, the robber cannot move to a vertex of $G-V(C)$ in order to avoid immediate capture. So for the next few turns, Cop~$1$ stays at $v$ and Cop~$2$ goes to a vertex of $C$ through the vertex $r$ and captures the robber.
	\end{proof}

	\section{On the class of $(P_5,$ diamond$)$-free graphs and $2K_1\cup K_2$-free graphs} \label{sec:diamondand2k1Uk2}
	
	Let $H$ be an induced subgraph of a graph $G$. A \emph{retraction} from $G$ to $H$ is a homomorphism from $G$ onto $H$ that maps every vertex of $H$ to itself. Formally, a retraction from $G$ to $H$ is a mapping $\phi:V(G)\rightarrow V(H)$ such that: $(1)$ $\phi(u)=u$ for every $u\in V(H)$ and $(2)$ if $xy\in E(G)$, then either $\phi(x)=\phi(y)$ or $\phi(x)\phi(y)\in E(H)$. We say that $H$ is a \emph{retract} of $G$ if there exists a retraction from $G$ to $H$. Note that if $G$ is a connected graph and $H$ is a retract of $G$, then $H$ is connected.
	
	In this section, we prove that the cop number of any $(P_5,$~diamond$)$-free graph and any $(P_5,2K_1\cup~K_2)$-free graph is at most~$2$. We use the following lemma to prove these results. We note that the lemma is an implication of a result of Berarducci and Intrigila \cite{berarduccioncopnumbers}. To make the paper self-contained, we give a proof of it. 
	
	\begin{lemma}[\cite{berarduccioncopnumbers}]\label{retractlem}
	Let $H$ be a retract of a connected graph $G$ such that every component of $G-V(H)$ is a complete graph. If $cop(H)\leq 2$, then $cop(G)\leq 2$. 
	\end{lemma}

	\begin{proof}
	Let $\phi$ be a retraction from $G$ to $H$ and $cop(H)\leq 2$. We show that two cops can capture the robber in the graph $G$ after a finite number of turns. At a turn, we say that the robber's image is at a vertex $u$ if the image of the robber's position is $u$ under $\phi$. 
	
	In the first turn, we place both the cops at some vertex of $H$. In the next turn, the robber chooses a vertex of $G$ as its position.  
	The cops first try to capture the robber's image by playing in the graph $H$. Note that since $\phi$ is a homomorphism, for any move of the robber in $G$, its image has a valid move in $H$, that is it either moves to an adjacent vertex in $H$ or stays at the same vertex of $H$. Since $H$ is connected and the cop number of $H$ is at most $2$, the cops capture the robber's image after a finite number of turns. So without loss of generality, we may assume that one of the cops, say Cop~$1$, is at the robber's image, the other cop is at any vertex of $H$, and it is now robber's turn to move. Since $\phi$ is an identity mapping on $H$, the robber is already captured if it is in $H$. So we may assume that the robber is in $G-V(H)$. Now Cop~$1$ follows the robber's image, that is after every turn of cops', Cop~$1$ is at the robber's image. So to avoid immediate capture, the robber must stay in $G-V(H)$. Hence the robber must move within a component of $G-V(H)$, say $C$. Then Cop~$2$ follows a path from its current position to a vertex of $C$. After a finite number of turns, Cop~$2$ and the robber are in $C$ and Cop~$1$ is at the robber's image. Recall that every component of $G-V(H)$ is a complete graph; in particular, $C$ is a complete graph. So the robber gets captured by Cop~$2$ implying that $cop(G)\leq 2$.
	\end{proof}
	
\subsection{$(P_5,$ diamond$)$-free graphs}

	In the following theorem, we prove that the cop number of any connected $(P_5,$ diamond$)$-free graph is at most~$2$.
	
	\begin{manualtheorem}{6.1.1}\label{p5diam}
	Let $G$ be a connected $(P_5,$~diamond$)$-free graph. Then $cop(G)\leq 2$.
	\end{manualtheorem}
	\begin{proof}
		For the sake of contradiction, assume that there exists a counterexample of the theorem. Let $G$ be a minimum counterexample of the theorem, that is $G$ is a connected $(P_5,$~diamond$)$-free graph with minimum number of vertices such that $cop(G)>2$. To proceed further, we first prove a series of claims.
		
\begin{claim}\label{thm:p5diamcl1}
If $u$ is a vertex of $G$, then every component of $G[N(u)]$ is a complete graph. Moreover, $G[N(u)]$ is a disconnected graph.
\end{claim}
\begin{proof}[Proof of Claim~\ref{thm:p5diamcl1}]
Let $u$ be a vertex of $G$. Since $G$ is diamond-free, $G[N(u)]$ is $P_3$-free and hence every component of $G[N(u)]$ is a complete graph. Now for the sake of contradiction, assume that $G[N(u)]$ is connected, that is it has only one component, say $G'$. Note that for any $w\in V(G')$, $N[u]\subseteq N[w]$ since $G'$ is a complete graph. Define a mapping $\phi:V(G)\rightarrow V(G)\setminus \{u\}$ that maps $u$ to $w$ for some $w\in V(G')$ and maps every vertex of $G-\{u\}$ to itself. Note that $\phi$ is a retraction from $G$ to $G-\{u\}$, that is $G-\{u\}$ is a retract of $G$. Moreover, since $G$ is a minimum counterexample, $cop(G-\{u\})\leq 2$. So by taking $H=G-\{u\}$ in Lemma~\ref{retractlem}, we have $cop(G)\leq 2$, a contradiction. Hence $G[N(u)]$ is a disconnected graph.
\end{proof}

\begin{claim}\label{thm:p5diamcl2}
If $u$ is a vertex of $G$, then for any vertex $v$ other than $u$, $N(v)\setminus N[u]\neq \emptyset$.
\end{claim}
\begin{proof}[Proof of Claim~\ref{thm:p5diamcl2}]
 Let $u$ be a vertex of $G$. For the sake of contradiction, assume that there exists a vertex $v$ other than $u$ such that $N(v)\setminus N[u]=\emptyset$, that is $N(v)\subseteq N[u]$. Define a mapping $\phi':V(G)\rightarrow V(G)\setminus \{v\}$ that maps $v$ to $u$ and maps every vertex of $G-\{v\}$ to itself. Note that $\phi'$ is a retraction from $G$ to $G-\{v\}$, that is $G-\{v\}$ is a retract of $G$. Moreover, since $G$ is a minimum counterexample, $cop(G-\{v\})\leq 2$. So by taking $H=G-\{v\}$ in Lemma~\ref{retractlem}, we have $cop(G)\leq 2$, a contradiction. Hence $N(v)\setminus N[u]\neq \emptyset$. 
\end{proof}

\begin{claim}\label{thm:p5diamcl3}
If $uv$ is an edge of $G$, then there exist vertices $y,z\in V(G)\setminus N[u]$ such that $\{u,v,y,z\}$ induces a $P_4$ in $G$. 
\end{claim}
\begin{proof}[Proof of Claim~\ref{thm:p5diamcl3}]
Let $uv$ be an edge of $G$. By Claim~\ref{thm:p5diamcl2}, $N(v)\setminus N[u]\neq\emptyset$. Let $G^*$ be a component of the graph induced by $N(v)\setminus N[u]$. By Claim~\ref{thm:p5diamcl1}, every component of the graph $G[N(v)]$ is a complete graph. Hence every component of the graph induced by $N(v)\setminus N[u]$ is a complete graph; in particular, $G^*$ is a complete graph. 

To prove the claim, it is sufficient to show the existence of a vertex $y\in V(G^*)$ that has a neighbor $z$ such that $z\notin N[\{u,v\}]$. For the sake of contradiction, assume that every neighbor of every vertex of $V(G^*)$ is in the set $N[\{u,v\}]$. Then $N(V(G^*))\subseteq N[\{u,v\}]$. Since $G^*$ is a component of the graph induced by $N(v)\setminus N[u]$, we have $N(V(G^*))\subseteq N[u]$. Define a mapping $\phi'':V(G)\rightarrow V(G)\setminus V(G^*)$ that maps every vertex of $G^*$ to $u$ and maps every vertex of $G-V(G^*)$ to itself. Note that $\phi''$ is a retraction from $G$ to $G-V(G^*)$, that is $G-V(G^*)$ is a retract of $G$. Moreover, since $G$ is a minimum counterexample, $cop(G-V(G^*))\leq 2$. Since $G^*$ is a complete graph, by taking $H=G-V(G^*)$ in Lemma~\ref{retractlem}, we have $cop(G)\leq 2$, a contradiction. Hence the claim holds. 
\end{proof}

			 Now we return to the proof of Theorem~\ref{p5diam}.  If $G$ is $K_4$-free, then by Theorem~\ref{p5k4}, $cop(G)\leq 2$, a contradiction. Hence $G$ must contain a subgraph isomorphic to $K_4$. Let $u$ be a vertex of any induced $K_4$ of $G$. Then there exists a component $G'$ of $G[N(u)]$ such that $|V(G')|\geq 3$. By Claim~\ref{thm:p5diamcl1}, the graph $G[N(u)]$ is disconnected and every component of $G[N(u)]$ is a complete graph. So there exists a component $G''$ of $G[N(u)]$ other than $G'$ and both $G'$ and $G''$ are complete graphs. Let $v$ be a vertex of $G''$. By Claim~\ref{thm:p5diamcl3}, there exist vertices $y,z\in V(G)\setminus N[u]$ such that $\{u,v,y,z\}$ induces a $P_4$ in $G$. Note that $y$ has at most one neighbor in $V(G')$; otherwise $G[V(G')\cup \{y,u\}]$ contains an induced diamond. Similarly, we can show that $z$ has at most one neighbor in $V(G')$. Since $G'$ and $G''$ are different components of $G[N(u)]$ and $v$ is a vertex of $G''$, $v$ does not have any neighbor in $V(G')$. Now since $|V(G')|\geq 3$, there exists a vertex $x\in V(G')$ such that $x$ is not adjacent to $v,y,$ and $z$. Then $\{x,u,v,y,z\}$ induces a $P_5$ in $G$, a contradiction. So we may conclude that such a graph $G$ does not exist. This completes the proof of Theorem~\ref{p5diam}. 
	\end{proof}
	
\subsection{$2K_1\cup K_2$-free graphs}
	
	 Let $G$ be a connected $2K_1\cup K_2$-free graph and $uv$ be an edge of $G$. If $\{u,v\}$ is a dominating set of $G$, then $cop(G)\leq 2$. If $\{u,v\}$ is not a dominating set of $G$, then there exists a vertex $z$ that is not adjacent to $u$ and $v$. Now since $G$ is $2K_1\cup K_2$-free, $\{u,v,z\}$ is a dominating set of $G$. So the cop number of $G$ is at most $3$. Therefore, the cop number of any $2K_1\cup K_2$-free graph is at most~$3$. Turcotte~\cite{jturcotte2k2freegraphs} showed the existence of $2K_1\cup K_2$-free graphs having the cop number~$3$ with computer aided graph search. So the cop number of the class of $2K_1\cup K_2$-free graphs is~$3$. Note that any complete multipartite graph is $K_1\cup K_2$-free. Now consider a $K_1\cup K_2$-free graph $G$. Since $G^c$ is $P_3$-free, $G^c$ is a disjoint union of complete graphs. So $G$ is a complete multipartite graph. Hence we may conclude that a graph $G$ is $K_1\cup K_2$-free if and only if $G$ is a complete multipartite graph.

\begin{manualtheorem}{6.2.1}\label{2k1uk2}
Let $G$ be a connected $(P_5,2K_1\cup K_2)$-free graph. Then $cop(G)\leq 2$.
\end{manualtheorem}
\begin{proof}
For the sake of contradiction, assume that there exists a counterexample of the theorem. Let $G$ be a minimum counterexample of the theorem, that is $G$ is a connected $(P_5,2K_1\cup K_2)$-free graph with minimum number of vertices such that $cop(G)>2$. Let $u$ be a vertex of $G$ and $G'$ be the graph $G-N[u]$. Since $G$ is $2K_1\cup K_2$-free, $G'$ is a $K_1\cup K_2$-free graph. Hence $G'$ is a complete multipartite graph, say with $k$~parts. We first show that $k\geq 2$. If possible, then let $k\leq 1$. If $k=0$, then $\{u\}$ is a dominating set of $G$ implying that $cop(G)=1$, a contradiction. So $k=1$. Then $G'$ consists of only isolated vertices. We show that two cops can capture the robber to obtain a contradiction to the fact that $cop(G)>2$. We place both the cops at $u$. The robber must choose a vertex of $G'$ to avoid immediate capture. While one cop stays at $u$, the robber cannot move since $G'$ consists of only isolated vertices. Then the second cop can go and capture the robber in the graph $G'$. Hence we have $k\geq 2$. Now if there exists a partite set $S$ of $G'$ such that $|S|=1$, then $\{u\}\cup S$ is a dominating set of $G$ implying that $cop(G)\leq 2$, a contradiction. So we may assume that every partite set of $G'$ has cardinality at least $2$.

\begin{claim} \label{2k1uk2-cl1}
	If $v$ is a neighbor of $u$, then $v$ has at most one non-neighbor in $V(G')$.
\end{claim}

\begin{proof}[Proof of Claim~\ref{2k1uk2-cl1}]
 Let $v$ be a neighbor of $u$. For the sake of contradiction, assume that $v$ has two non-neighbors in $V(G')$, say $x$ and $y$. Since $\{u,v,x,y\}$ does not induce a $2K_1\cup K_2$ in $G$, $x$ and $y$ are in different partite sets of the complete multipartite graph $G'$. Recall that every partite set of $G'$ has cardinality at least $2$. Let $x'$ be a vertex of $G'$ other than $x$ such that $x$ and $x'$ are in the same partite set of $G'$. Note that $v$ is adjacent to $x'$; otherwise $\{u,v,x,x'\}$ induces a $2K_1\cup K_2$ in $G$. Since $y$ and $x'$ are in different partite sets of the complete multipartite graph $G'$, $y$ is adjacent to $x'$. Similarly, $y$ is adjacent to $x$. Then $\{u,v,x',y,x\}$ induces a $P_5$ in $G$, a contradiction. So $v$ has at most one non-neighbor in $V(G')$. 
\end{proof}

\begin{claim}\label{2k1uk2-cl2}
If $z$ is a vertex of $G'$, then $z$ has a non-neighbor in $N(u)$.
\end{claim}
\begin{proof}[Proof of Claim~\ref{2k1uk2-cl2}]
Let $z$ be a vertex of $G'$. For the sake of contradiction, assume that $z$ is adjacent to every vertex of $N(u)$, that is $N(u)\subseteq N(z)$. Define a mapping $\phi: V(G)\rightarrow V(G) \setminus \{u\}$ that maps $u$ to $z$ and maps every vertex of $G-\{u\}$ to itself. Note that $\phi$ is a retraction from $G$ to $G-\{u\}$. Moreover, since $G$ is minimum counterexample, $cop(G-\{u\})\leq 2$. By taking $H=G-\{u\}$ in Lemma~\ref{retractlem}, we have $cop(G)\leq 2$, a contradiction. So $z$ has non-neighbor in $N(u)$.
\end{proof}

Now we show that two cops can capture the robber in $G$. Let $v$ be a neighbor of $u$. In the first turn, we place Cop~$1$ at $u$ and Cop~$2$ at $v$. To avoid immediate capture, the robber must choose a vertex $x$ that is not adjacent to $u$ and $v$.  Since $x\notin N[u]$, $x\in G'$. Again since $x$ is not adjacent to $v$, by Claim~\ref{2k1uk2-cl1}, $x$ is the only non-neighbor of $v$ in $V(G')$, that is $v$ is adjacent to every vertex of $V(G')\setminus \{x\}$. Since every partite set of $G'$ has cardinality at least $2$, there exists a vertex $x'$ other than $x$ such that $x$ and $x'$ are in the same parite set of $G'$.
By Claim~\ref{2k1uk2-cl2}, $x'$ has a non-neighbor $r$ in $N(u)$. Again by Claim~\ref{2k1uk2-cl1}, $r$ is adjacent to every vertex of $G'$ except $x'$. In particular, $r$ is adjacent to $x$.

In the next turn, Cop~$1$ moves to $r$ and Cop~$2$ moves to $u$. To avoid immediate capture, the robber must move to a vertex that is not adjacent to $r$ and $u$. Since $x'$ is the only vertex that is not adjacent to $r$ and $u$, the robber must move to $x'$. This is not possible since $x$ and $x'$ are in the same partite set of the multipartite graph $G'$. So the robber cannot escape from $x$ and gets captured implying that $cop(G)\leq 2$, a contradiction. So such a graph $G$ does not exist. This completes the proof of Theorem~\ref{2k1uk2}.
\end{proof}

	\section{Conclusion}
	
	In this paper, we obtained strategies using two cops to capture the robber in a $(P_5,H)$-free graph, where $H\in \{C_4$, $C_5$, claw, diamond, paw, $K_4$, $2K_1\cup K_2$, $K_3\cup K_1$, $P_3\cup P_1\}$. On the other hand, $P_4$-free graphs and $2K_2$-free graphs are already known to have the cop number at most~$2$. By including these results, we conclude that the cop number of $(P_5,\mathcal{H})$-free graphs is at most~$2$, where $\mathcal{H}$ is any graph on $4$~vertices with at least one edge. Moreover, $C_4$ and $C_5$ have the cop numbers~$2$ and at least one of these belongs to the class of $(P_5,\mathcal{H})$-free graphs. Thus the cop number of this class is $2$. Note that Conjecture~\ref{conj} remains open even for $t=5$. Even the question whether $2$~cops are sufficient to capture the robber in a $P_5$-free graph with independence number at most~$3$, remains open. Although we have focused on the subclasses of $P_5$-free graphs, the methods we have used may be applied to obtain the cop number of some other graph classes with forbidden induced subgraphs. 
	
%

	
%
	
\end{document}